\newif\ifarxiv\arxivfalse
 \arxivtrue

\RequirePackage[dvipsnames]{xcolor}
\ifarxiv
	\documentclass[10pt, twoside, reqno]{amsart}
\else
	\RequirePackage{amsmath}
	\documentclass[[smallextended,referee,envcountsect]{svjour3}
\fi
	\PassOptionsToPackage{noend}{algpseudocode}
	\PassOptionsToPackage{bookmarksdepth=4}{hyperref}
	
	\ifarxiv\else
		\makeatletter
			\let\cl@chapter\undefined
		\makeatother

		\smartqed
	\fi

	\usepackage{booktabs}
	\usepackage{multirow}
	\usepackage[%
		eqreset=section,
		thmreset=section,
	]{myPreamble}

	\usepackage{todonotes}

	\CreateNewTheorem[Fact]{fact}[dummythm]{Fact}
	
	\newcommandx\p[3][1=\gamma,3=h]{\prox_{#1#2}^{#3}}
	
	\newcommand\C{\mathcal C}
	\makeatletter
		\newcommand\D[1][h_i]{%
			\@ifnextchar_{\operatorname D}{%
				\ifstrempty{#1}{\operatorname D}{\operatorname D_{#1}}%
			}
		}
	\makeatother
	\newcommandx\FBE[2][1={f,g},2={\gamma h_i}]{\varphi_{#2}^{#1}}
	\newcommandx\T[2][2={\nicefrac{h_i}{\gamma_i^k}},1={ }]{\operatorname T_{#2}^{#1}}
	\newcommandx\M[2][2={\nicefrac{h_i}{\gamma_i^k}},1={ }]{\operatorname{\mathcal{M}}_{#2}^{#1}}
	\newcommandx\TG[2][1={f,g},2={\gamma^k h_i}]{\operatorname T_{#2}^{#1}}
	\newcommandx\RGH[2][1={f,g},2={\gamma h_i}]{\operatorname R_{#2}^{#1}}
	\let\proj\relax
	\DeclareMathOperator\proj{Proj}
	
		\newcommand\refBIBPA[1][ (\cref{alg:bibpa})]{%
			\hyperref[alg:bibpa]{{\tt\textsc{BIBPA}}}#1%
		}%

\newcommand{\TheKeywords}{%
	Nonsmooth nonconvex optimization,
	Block Bregman proximal algorithm,
	Inertial effects,
    block relative smoothness,
    Symmetric nonnegative matrix tri-factorization%
}
\newcommand{\TheSubjclass}{%
	90C06, 
	90C25, 
	90C26, 
	49J52, 
	49J53.
}
\newcommand{\TheTitle}{%
	A block inertial Bregman proximal algorithm for nonsmooth nonconvex problems with application to symmetric nonnegative matrix tri-factorization%
}
\newcommand{\TheShortTitle}{%
	A block inertial Bregman proximal algorithm%
}
\newcommand{\TheShortAuthor}{%
	M. Ahookhosh, L.T.K. Hien, N. Gillis, and P. Patrinos%
}
\newcommand{\TheFunding}{%
	MA and PP acknowledge the support by the \emph{Research Foundation Flanders (FWO)} research projects G086518N and G086318N;
	\emph{Research Council KU Leuven} C1 project No. C14/18/068;
	\emph{Fonds de la Recherche Scientifique - FNRS and the Fonds Wetenschappelijk Onderzoek - Vlaanderen} (FWO) under EOS project no 30468160 (SeLMA).
	NG  also acknowledges the support by the European Research Council (ERC starting grant no 679515).
}
\newcommand{\TheAddressMons}{%
	Department of Mathematics and Operational Research, Faculté polytechnique, Université de Mons.
	Rue de Houdain 9,
	7000 Mons,
	Belgium.
}
\newcommand{\TheAbstract}{%
	We propose \refBIBPA[], a block inertial Bregman proximal algorithm for minimizing the sum of a block relatively smooth function (that is, relatively smooth concerning each block) and block separable nonsmooth nonconvex functions. We prove that the sequence generated by \refBIBPA[] subsequentially converges to critical points of the objective under standard assumptions, and globally converges when the objective function is additionally assumed to satisfy the Kurdyka-{\L}ojasiewicz (K{{\L}}) property. We also provide the convergence rate when the objective satisfies the {\L}ojasiewicz inequality. We apply BIBPA to the symmetric nonnegative matrix tri-factorization (SymTriNMF) problem, where we propose kernel functions for SymTriNMF and provide closed-form solutions for subproblems of \refBIBPA[]. 
}

\ifarxiv
	\title[\TheShortTitle]{\TheTitle}
	\author[\TheShortAuthor]{%
		Masoud Ahookhosh\textsuperscript{1},\ 
		Le Thi Khanh Hien\textsuperscript{2}.\ 
		Nicolas Gillis\textsuperscript{2},\ and\ 
		Panagiotis Patrinos\textsuperscript{1}%
	}
	\thanks{\textsuperscript{1}\TheAddressKU}
	\thanks{\textsuperscript{2}\TheAddressMons}
	\thanks{\TheFunding}
	\keywords{\TheKeywords}
	\subjclass{\TheSubjclass}

\begin{document}

		\begin{abstract}
			\TheAbstract
		\end{abstract}

		\maketitle

\else

	\begin{document}
	\journalname{Journal of Optimization Theory and Applications}

	\title{\TheTitle}
	\titlerunning{\TheShortTitle}

	\author{%
		Masoud Ahookhosh\and
		Le Thi Khanh Hien\and
		Nicolas Gillis\and
		Panagiotis Patrinos%
	}
	\authorrunning{\TheShortAuthor}

	\institute{%
		M. Ahookhosh and P. Patrinos%
		\at
			\TheAddressKU.
			{\tt
				\{%
					\href{mailto:masoud.ahookhosh@esat.kuleuven.be}{masoud.ahookhosh},%
					\href{mailto:panos.patrinos@esat.kuleuven.be}{panos.patrinos}%
				\}\href{mailto:masoud.ahookhosh@esat.kuleuven.be,panos.patrinos@esat.kuleuven.be}{@esat.kuleuven.be}%
			}%
		\and
		L. T. K. Hien and N. Gillis
		\at
			\TheAddressMons.
			{\tt
				\{%
					\href{mailto:ThiKhanhHien.LE@umons.ac.be>}{ThiKhanhHien.LE},%
					\href{mailto:nicolas.gillis@umons.ac.be}{nicolas.gillis}%
				\}%
				\href{mailto:ThiKhanhHien.LE@umons.ac.be,nicolas.gillis@umons.ac.be}{@umons.ac.be}%
			}%
	}%

	\maketitle
    \vspace{-5mm}
	\begin{abstract}
		\TheAbstract
		\keywords{\TheKeywords}%
		\subclass{\TheSubjclass}%
	\end{abstract}
\fi

    \vspace{-4mm}	
	\section{Introduction} \label{sec:intro}
	This paper is concerned with the minimization of the sum of a block relatively smooth (see \Cref{def:mbRelSmooth}), and a block separable (nonsmooth) nonconvex function. 
	Although this problem has a simple structure, it covers a broad range of optimization problems arising in signal and image processing, machine learning, and inverse problems.
	In our block-structured nonconvex setting, the most common class of methodologies is \DEF{first-order} ones, where the central to their convergence analysis is the so-called \DEF{descent lemma} in both the Euclidean setting (e.g., \cite{ahookhosh2019accelerated,bauschke2017convex,beck2017first,nesterov2003introductory,nesterov2015universal}) and the non-Euclidean one (e.g., \cite{bauschke2016descent,lu2018relatively,van2017forward}). While for the Euclidean case, the descent lemma is guaranteed if the function has Lipschitz continuous gradients, in the non-Euclidean setting it holds for \DEF{relatively smooth} functions encompassing the class of smooth functions with Lipschitz gradients. 
	
	In the \DEF{Euclidean setting}, there are large number alternating minimization algorithms for handling our structured problem such as block coordinate methods \cite{beck2015cyclic,beck2013convergence,nesterov2012efficiency,tseng2001convergence,tseng2009coordinate} and Gauss-Seidel methods \cite{auslender1976optimisation,bertsekas1989parallel,grippo2000convergence}, proximal alternating minimization \cite{attouch2010proximal,attouch2007new}, and proximal alternating linearized minimization \cite{bolte2014proximal,pock2016inertial,shefi2016rate}. In the \DEF{non-Euclidean setting}, several algorithms have been proposed, namely, Bregman forward-backward splitting 
	 \cite{ahookhosh2019bregman,bauschke2019linear,bauschke2016descent,bolte2018first,lu2018relatively,teboulle2018simplified}, accelerated Bregman forward-backward splitting \cite{gutman2018unified,hanzely2018accelerated}, stochastic mirror descent methods \cite{hanzely2018fastest}, Bregman proximal alternating linearized minimization \cite{ahookhosh2019multi}.
	
	In order to establish the global convergence of generic algorithms for (nonsmooth) nonconvex problems, one needs to assume that the celebrated Kurdyka-{\L}ojasiewicz inequality (see \Cref{def:klFunctions1}) is satisfied as a feature of the underlying problem's class. The earliest abstract convergence theorem was  introduced by Attouch et al. \cite{attouch2013convergence} and by Bolte et al. \cite{bolte2014proximal}, relying on the following conditions that an algorithm should satisfy: 
	(i)~\DEF{sufficient decrease condition of the cost function}; 
	(ii)~\DEF{subgradient lower bound of iterations gap}; 
	(iii)~\DEF{subsequential convergence}. These conditions are shown to be satisfied by many algorithms~\cite{attouch2013convergence}. In \cite{bolte2014proximal}, these conditions were extended for proximal alternating linearized minimization. In the case of inertial proximal point algorithms \cite{ochs2014ipiano,pock2016inertial}, it was shown that some Lyapunov function satisfies the sufficient decrease condition, which leads to a generalization of the abstract convergence theorem. A generalization of this theorem was introduced for variable metric algorithms in \cite{frankel2015splitting}, which has been recently extended for inertial variable metric algorithms \cite{ochs2019unifying}. In this paper, we show that the results of \cite{ochs2019unifying} can cover the global convergence of algorithms in non-Euclidean settings.

	\vspace{-5mm}
	\subsection{Contribution}\label{sec:contribution} 
	
	 Our contribution is twofold: 
	 \begin{enumerate}[wide, labelwidth=!, labelindent=0pt] 
	     \item (\DEF{Block inertial Bregman proximal algorithm}) We introduce \refBIBPA[], a block generalization of the Bregman proximal gradient method \cite{bolte2014proximal} \emph{with an inertial force}. 
	     We extend the notion of relative smoothness \cite{bauschke2016descent,lu2018relatively,van2017forward} to its block version (with different kernels for each block) to support our structured nonconvex problems. 
	     Notably, these kernel functions are block-wise convex, a property that does not necessarily imply their joint convexity for all blocks. 
	     Unlike the global convergence theorem in \cite{attouch2013convergence,bolte2014proximal} that verifies the \DEF{sufficient decrease condition} and  \DEF{subgradient lower bound of iterations gap} on the cost function, for \refBIBPA[] these properties hold for a \DEF{Lyapunov function} including Bregman terms (see the equation \eqref{eq:LyapunovFunc}). 
	     Then, the global convergence of \refBIBPA[] is studied under the K{{\L}} property, and its convergence rate is studied for {\L}ojasiewicz-type K{{\L}} functions. 
	   
	     \item (\DEF{Globally convergent scheme for solving the SymTriNMF problem})  With appropriate selection of kernel functions for Bregman distances, it turns out that the objective of the \DEF{symmetric nonnegative matrix tri-factorization} (SymTriNMF) problem is block relatively smooth, and the corresponding subproblems can be solved in closed form, an important property when dealing with machine learning problems that include a large number of variables. To the best of our knowledge, \refBIBPA[] is the first scheme with a rigorous theoretical guarantee of convergence for the SymTriNMF problem.
	 \end{enumerate}

	\vspace{-9mm}
	\subsection{Related works}\label{sec:relativeWorks} 
	There are three papers \cite{ahookhosh2019multi,wang2018block,zhang2019inertial} that are closely related to this paper.
	In \cite{ahookhosh2019multi}, we introduced a multi-block relative smoothness condition that exploits a single kernel function for all blocks, while in the current paper we assume a block relative smoothness condition allowing a different kernel function for each block. Moreover, our algorithm \refBIBPA[] involves dynamic step-sizes and inertial terms for each block that makes our derivation and analysis different from those of \cite{ahookhosh2019multi}. Beside of the algorithmic differences with \cite{wang2018block}, we use nonseparable (nonconvex) kernels as apposed to the separable convex kernel used in \cite{wang2018block} for each block.
	An inertial Bregman proximal gradient algorithm was presented in \cite{zhang2019inertial} for composite minimization that does not support our block structure nonconvex problems and therefore is different in derivation and analysis concerning our work.

	\vspace{-8mm}
	\subsection{Organization}\label{sec:organization} 
	The remainder of this paper is organized as follows.  While \Cref{sec:probRS} discusses the problem statement and the block relative smoothness, \Cref{sec:unipa} introduces and analyzes a block inertial Bregman proximal algorithm (\refBIBPA[]). In \Cref{sec:SNMF}, it is shown the \refBIBPA[]'s subproblems are solved in closed form. 
	Some conclusion are delivered in \Cref{sec:conclusion}.
	
	\section{Problem statement and block relative smoothness}\label{sec:probRS}
	We consider the structured nonsmooth nonconvex minimization problem
	\begin{equation}\label{eq:P}
		\minimize_{\bm x\in \overline{C}}~~\Phi(\bm x)\equiv f(\bm x)+\sum_{i=1}^N g_i(x_i),
	\end{equation}
	where $C$ is a nonempty, convex, and open set in $\R^n$ and $\overline{C}$ denotes its closure.  Setting  $n=\sum_{i=1}^N n_i$ and $i=1,\ldots,N$, we assume the following hypotheses:
	\begin{ass}[requirements for composite minimization \eqref{eq:P}]\label{ass:basic:fgh}~
	\begin{enumeratass}
		\item\label{ass:basic:g}%
			\(\func {g_i}{\R^{n_i}}{\Rinf\coloneqq\R\cup\set\infty}\) is proper and lower semicontinuous (lsc); 
			\item\label{ass:basic:h}%
			\(\func{h_i}{\R^{n}}{\Rinf}\) is $i$-th block Legandre,  $\interior\dom h_1=\ldots=\interior\dom h_N$, $\overline{C}\subseteq\overline{\dom h_1}$, and $\dom g\cap C\neq\emptyset$ with $g:=\sum_{i=1}^N g_i$;
		\item\label{ass:basic:f}%
			\(\func{f}{\R^{n}}{\Rinf}\) is $\C^1(\interior\dom h_1)$ and \DEF{$(L_1,\ldots,L_N)$-smooth relative to $(h_1,\ldots,h_N)$};
		\item\label{ass:basic:argmin}%
	  $\argmin \set{\Phi(x) \mid \bm x\in \overline{C}}\neq\emptyset$.
	\end{enumeratass}
	\end{ass}
	
	
	\vspace{-5mm}
	\subsection{Notation}\label{sec:notation}
	We denote by $\Rinf\coloneqq\R\cup\set{\infty}$ the extended-real line. 
	We use boldface lower-case letters (e.g., $\bm x$, $\bm y$, $\bm z$) for vectors in $\R^{n}$ and use normal lower-case letters (e.g., $z_i$, $x_i$, $y_i$) for vectors in $\R^{n_i}$, for $n_i\in \N$. For the identity matrix $I_n$, we set $U_i\in\R^{n\times n_i}$ such that
	 $I_n=(U_1,\ldots,U_N)\in\R^{n\times n}$. 
	The set of cluster points of \(\seq{\bm x^k}\) is denoted as \(\omega(\bm x^0)\).
	A function $\func{f}{\R^n}{\Rinf}$ is \DEF{proper} if $f>-\infty$ and $f\not\equiv\infty$, in which case its \DEF{domain} is defined as the set
	$\dom f\coloneqq\set{\bm x\in\R^n}[f(\bm x)<\infty]$. 
	A vector $\bm v\in\partial f(\bm x)$ is a \DEF{subgradient} of $f$ at $\bm x$, and the set of all such vectors is called the \DEF{subdifferential} $\partial f(\bm x)$ \cite[Definition 8.3]{rockafellar2011variational}, i.e. 
	\begin{align*}
		\partial f(\bm x)
	{}={} &
		\set{\bm v\in\R^n}[
			\exists\seq{\bm x^k,\bm v^k}~\text{s.t.}~ \bm x^k\to \bm x,~f(\bm x^k)\to f(\bm x),~
			\widehat\partial f(\bm x^k)\ni \bm v^k\to \bm v
		],
	\shortintertext{%
		where $\widehat\partial f(\bm x)$ is the set of \DEF{regular subgradients} of $f$ at $\bm x$, namely
	}
		\widehat\partial f(\bm x)
	{}={} &
		\set{\bm v\in\R^n}[
			f(\bm z)\geq f(\bm x){}+{}\innprod{\bm v}{\bm z-\bm x}{}+{}o(\|\bm z-\bm x\|),~
			\forall \bm z\in\R^n\vphantom{\seq{\bm x^k}}
		].
	\end{align*} 
	
	\vspace{-2mm}
	\subsection{Block relative smoothness}\label{sec:relSmooth}
	We first describe the notion of \DEF{block relative smoothness}, which is an extension of the relative smoothness \cite{bauschke2016descent,lu2018relatively}. To this end, we introduce the notion of \DEF{block kernel} functions, which coincides with the classical one (cf. \cite[Definition 2.1]{ahookhosh2019bregman}) for $N=1$.
	
	\begin{defin}[$i$-th block convexity and kernel function]\label{def:kernel}%
		Let \(\func{h}{\R^n}{\Rinf}\) be a proper and lower semicontinuous (lsc) function with \(\interior\dom h\neq\emptyset\) and such that  $h\in\C^1(\interior\dom h)$.
		For a fixed vector $\bm x\in\R^n$ and $i\in\set{1,\ldots,N}$, we say that \(h\) is
		\begin{enumerate}
		\item 
		     \DEF{$i$-th block (strongly/strictly) convex} if the function $h(\bm x+U_i(\cdot -x_i))$ is (strongly/strictly) convex for all $\bm x\in \dom h$;
		\item
			a \DEF{$i$-th block kernel function} if $h$ is $i$-th block convex and $h(\bm x+U_i(\cdot -x_i))$ is $1$-coercive for all $\bm x\in \dom h$, i.e., $\lim_{\|z\|\to\infty}\tfrac{h(\bm x+U_i(z -x_i))}{\|z\|}=\infty$;
		\item
			\DEF{$i$-th block essentially smooth}, if for every sequence $\seq{\bm x^k}\subseteq\interior\dom h$ converging to a boundary point of $\dom h$, we have $\|\nabla_i h(\bm x^k)\|\to\infty$;
		\item\label{def:kernel4}
			\DEF{$i$-th block Legendre} if it is $i$-th block essentially smooth and $i$-th block strictly convex.
		\end{enumerate}
	\end{defin}
	
	Let $\func{h}{\R^n}{\Rinf}$ be a Legendre function. Then, the classical definition of \DEF{Bregman distances} (cf. \cite{bregman1967relaxation}) leads to the function $\func{\D_{h}}{\R^n\times\R^n}{\Rinf}$ given by
	\begin{equation}\label{eq:bregman}
			\D_{h}(\bm y,\bm x)
		{}\coloneqq{}
			\begin{ifcases}
				h(\bm y)-h(\bm x)-\innprod{\nabla h(\bm x)}{\bm y-\bm x} &\bm y\in\dom h, \bm x\in\interior\dom h
			\\
				\infty\otherwise.
			\end{ifcases}
	\end{equation}
	However, in the remainder of this paper, we extend this definition for the cases that $h$ is only an $i$-th block Legendre function.
	Fixing all blocks except the $i$-th one, the Bregman distance \eqref{eq:bregman} will reduce to
	$   \D_{h}(\bm x+U_i(y_i-x_i),\bm x)=h(\bm x+U_i(y_i-x_i))-h(\bm x)-\langle\nabla_i h(\bm x),y_i-x_i\rangle,
	$ which measures the proximity between $\bm x+U_i(y_i-x_i)$ and $\bm x$ with respect to the $i$-th block of variables.
	Moreover, the kernel \(h\) is $i$-th block convex if and only if \(\D_{h}(\bm x+U_i(y_i-x_i),\bm x)\geq 0\) for all \(\bm x+U_i(y_i-x_i)\in\dom h\) and \(\bm x\in\interior \dom h\). 
	Note that if \(h\) is $i$-th block strictly convex, then $\D_{h}(\bm x+U_i(y_i-x_i),\bm x)= 0$ if and only if $x_i=y_i$. 
	
	We are now in a position to present the notion of \DEF{block relative smoothness}, which is the central tool for our analysis in the next section. 
	
	\begin{defin}[block relative smoothness]
	\label{def:mbRelSmooth}
	For $i\in[N]$, let \(\func{h_i}{\R^n}{\Rinf}\) be $i$-th block kernel functions and let \(\func{f}{\R^n}{\Rinf}\) be a proper and lsc function. 
	If there exists $L_i> 0$, $i\in [N]$, such that 
	$L_ih_i(\bm x+U_i(z-x_i))-f(\bm x+U_i(z-x_i))$
	are convex for all $\bm x, \bm x+U_i(z-x_i)\in \interior\dom h_i$, then, $f$ is called \DEF{$(L_1,\ldots,L_N)$-smooth relative to $(h_1,\ldots,h_N)$}.
	\end{defin}
	
	Note that if $N=1$, the block relative smoothness is reduced to standard relative smoothness \cite{bauschke2016descent,lu2018relatively}. 
	If \(f\) is \(L\)-Lipschitz continuous, then both \(\nicefrac{L}{2}\|\cdot\|^2-f\) and \(\nicefrac{L}{2}\|\cdot\|^2+f\) are convex, i.e., the relative smoothness of \(f\) generalizes the notions of Lipschitz continuity.
	
	
	\begin{prop}[characterization of block relative smoothness] 
	\label{fac:relSmoothEqvi}
	    For $i=1,\ldots,N$, let \(\func{h_i}{\R^n}{\Rinf}\) be $i$-th block kernels and let \(\func{f}{\R^n}{\Rinf}\) be a proper 
	    lsc function and \(f\in\mathcal{C}^1\). Then, the following statements are equivalent:
	    \begin{enumerateq}
	    \item \label{fac:relSmoothEqvi1}
	    $(L_1,\ldots,L_N)$-smooth relative to $(h_1,\ldots,h_N)$; 
	    \item \label{fac:relSmoothEqvi2}
	    for all \((\bm x,\bm y)\in\interior\dom h_i\times \interior\dom h_i\) and $i=1,\ldots,N$, 
	    \begin{equation}\label{eq:upperIneq}
	        f(\bm x+U_i(y_i-x_i)) \leq f(\bm x)+\innprod{\nabla_i f(\bm x)}{y_i-x_i}+L_i \D(\bm x+U_i(y_i-x_i),\bm x);
	    \end{equation}
	    \item  \label{fac:relSmoothEqvi3}
	    for all \((\bm x,\bm y)\in\interior\dom h_i\times \interior\dom h_i\) and $i=1,\ldots,N$,
	    \begin{equation}\label{eq:upperEqvi}
	        \innprod{\nabla_i f(\bm x)-\nabla_i f(\bm y)}{x_i-y_i}\leq L_i \innprod{\nabla_i h_i(\bm x)-\nabla_i h_i(\bm y)}{x_i-y_i};
	    \end{equation}
	    \item \label{fac:relSmoothEqvi4}
	    if \(f\in\mathcal{C}^2(\interior\dom f)\) and 
	    \(h\in\mathcal{C}^2(\mathrm{\bf int}\dom h_i)\), then 
	    \begin{equation}\label{eq:descentEqviC2}
	        L_i \nabla_{x_ix_i}^2 h_i(\bm x)-\nabla_{x_ix_i}^2 f(\bm x)\succeq 0,\quad \forall \bm x\in\interior\dom h_i,~i=1,\ldots,N.
	    \end{equation}
	    \end{enumerateq}
	\end{prop}
	
	\begin{proof}
	The proof is a straightforward extension of those given in~\cite[Proposition 1.1]{lu2018relatively}, by fixing all the blocks except one of them.
	\end{proof}
	\subsection{Motivating example: symmetric nonnegative matrix tri-factorization}\label{sec:snmtf}
	We consider a symmetric matrix $X\in\R^{m\times m}$ and aim to decompose it in the form $X=UVU^T$, where $U\in\R_+^{m\times r}$ and $V\in\R_+^{r\times r}$. This translates to the minimization of $\tfrac{1}{2}\| X -UVU^T\|_F^2$ for $U, V\geq 0$, leading to the unconstrained problem
	\begin{equation}
	\label{eq:usnmtf}  
	\min_{U\in\R^{m\times r},V\in\R^{r\times r}} \tfrac{1}{2}\| X -UVU^T\|_F^2+\delta_{U\geq 0}+\delta_{V\geq 0}.
	\end{equation}
	
	\begin{prop}[block relative smoothness of SymTriNMF objective]
	\label{pro:relSmoothSNMF0}
	    Let functions $\func{h_1}{\R^{m\times r}\times\R^{r\times r}}{\Rinf}$ and $\func{h_2}{\R^{m\times r}\times\R^{r\times r}}{\Rinf}$ be strongly convex kernel functions as
	    \begin{align}
	        &\label{eq:kernelh4h21} h_1(U,V):=\tfrac{a_1}{4} \|V\|_F^2\|U\|_F^4+\tfrac{b_1}{2}\left(\|X\|_F \|V\|_F+\varepsilon_1\right) \|U\|_F^2, \\
	        &\label{eq:kernelh4h22} h_2(U,V):=\tfrac{a_2}{2} \left(\|U\|_F^4+\varepsilon_2\right)\|V\|_F^2.
	    \end{align}
	   with $\varepsilon_1, \varepsilon_2>0$. Then the function  $\func{f}{\R^{m\times r}\times \R^{r\times r}}{\Rinf}$ given by $f(U, V):=\tfrac{1}{2}\|X-UVU^T\|_F^2$ is $(L_1,L_2)$-smooth relative to $(h_1,h_2)$ with
	    \begin{equation}\label{eq:l1l2UpperBounds}
	        L_1\geq \max\set{\tfrac{6}{a_1},\tfrac{2}{b_1}}, 
	        \quad 
	        \text{ and } 
	        \quad 
	        L_2\geq \tfrac{1}{a_2}.
	    \end{equation}
	\end{prop}
	
	\begin{proof}
	    Plugging the partial derivative $\nabla_U f(U,V)=-XUV^T-X^TUV+UVU^TUV^T+UV^TU^TUV$ into the definition of directional derivative, we obtain
	    \begin{align*}
	    \begin{array}{ll}
	        \nabla_{UU}^2 f(U,V)Z&= -2XZV^T+UVU^TZV+UVZ^TUV+ZVU^TUV^T\\
	        &~~~+UV^TU^TZV+UV^TZ^TUV+ZV^TU^TUV,
	    \end{array}
	    \end{align*}
	    which consequently leads to
	    \begin{align*}
	     \begin{array}{ll}
	        \innprod{Z}{\nabla_{UU}^2 f(U,V)Z}&= -2\innprod{Z}{XZV^T}+\innprod{Z}{UVU^TZV}+\innprod{Z}{UVZ^TUV}+\innprod{Z}{ZVU^TUV^T}\\
	        &~~~+\innprod{Z}{UV^TU^TZV}+\innprod{Z}{UV^TZ^TUV}+\innprod{Z}{ZV^TU^TUV}\\
	        &\leq \left(2\|X\|~\|V\|+6\|U\|^2\|V\|^2\right)\|Z\|_F^2.
	        \end{array}
	    \end{align*}
	    On the other hand, from $\nabla_U h_1(U,V)=\left(a_1 \|U\|_F^2 \|V\|_F^2+b_1 \left(\|X\|_F~\|V\|_F+\varepsilon_1\right) \right)U$, we have
	    \begin{align*}
	        \nabla_{UU}^2 h_1(U,V)Z= \left(2a_1\|V\|^2\innprod{U}{Z} \right)U+\left(a_1 \|V\|_F^2\|U\|_F^2+b_1\left(\|X\|_F~\|V\|_F+\varepsilon_1\right) \right)Z,
	    \end{align*}
	    implying that
	\begin{align*}
	     \begin{array}{ll}
	        \innprod{Z}{\nabla_{UU}^2 h_1(U,V)Z}&\geq
	        \left(a_1 \|V\|_F^2\|U\|_F^2+b_1\left(\|X\|_F~\|V\|_F+\varepsilon_1\right) \right)\|Z\|_F^2\\
	        &\geq
	        \left(a_1 \|V\|^2\|U\|^2+b_1\left(\|X\|~\|V\|+\varepsilon_1\right) \right)\|Z\|_F^2.
	        \end{array}
	    \end{align*}
	    Therefore, the inequality 
	    \begin{align*}
	        \innprod{Z}{(L_1\nabla_{UU}^2 h_1(U,V)&-\nabla_{UU}^2 f(U,V))Z}\\
	        &\geq \left((L_1 a_1-6)\|V\|^2\|U\|^2+(L_1 b_1-2)\|X\|~ \|V\|+\varepsilon_1 L_1 \right)\|Z\|_F^2
	        \geq 0
	    \end{align*}
	    holds if $L_1 a_1-6\geq 0$ and $L_1 b_1-2\geq 0$, as claimed.
	    
	    It follows from $\nabla_V f(U,V)=U^TXU+U^TUVU^TU$ that
	    \begin{align*}
	        \nabla_{VV}^2 f(U,V)Z= \lim \frac{\nabla_U f(U+tZ,V)-\nabla_U f(U,V)}{t}=U^TUZU^TU,
	    \end{align*}
	    leading to the inequality
	    $\innprod{Z}{\nabla_{VV}^2 f(U,V)Z}= \innprod{Z}{U^TUZU^TU}\leq \|U\|^4\|Z\|_F^2$. Now, using 
	    $\nabla_V h_2(U,V)=a_2\left(\|U\|^4+\varepsilon_2\right) V$, we get $\innprod{Z}{\nabla_{VV}^2 h_2(U,V)Z}=a_2 \left(\|U\|^4+\varepsilon_2\right)\|Z\|_F^2$, i.e.,
	    \begin{align*}
	        \innprod{Z}{(L_2\nabla_{VV}^2 h_2(U,V)-\nabla_{VV}^2 f(U,V))Z}=\left((L_2a_2-1)\|U\|^4+\varepsilon_2 L_2\right)\|Z\|_F^2\geq 0
	    \end{align*}
	    if $L_2 a_2-1\geq 0$, giving our desired results.
	\end{proof}
	
	\vspace{-4mm}	
    \section{Block inertial Bregman proximal algorithm}\label{sec:unipa}
	This section discusses our algorithm, starting from the prox-boundedness extension \cite{rockafellar2011variational}.
	
	\begin{defin}[block prox-boundedness] 
	\label{def:bregman} 
	    A function $\func{g}{\R^{n}}{\Rinf}$ is \DEF{block prox-bounded} if for each $i\in \set{1,\ldots,N}$ there exists $\gamma_i>0$ and $\bm x\in\R^n$ such that 
	    \[
	        \inf_{z\in\R^{n_i}} \set{g(\bm x+U_i(z-x_i)) +\tfrac{1}{\gamma_i}\D(\bm x+U_i(z-x_i),\bm x)}>-\infty.
	    \]
	The supremum of the set of all such $\gamma_i$ is the threshold $\gamma_{i,g}^h$ of the block prox-boundedness, 
	\begin{equation}\label{eq:dProxBound}
	    \gamma_{i,g}^{h_i}:=\sup_{\gamma_i>0} \set{\gamma_i: \exists \bm x\in\R^n, 
	    \inf_{z\in\R^{n_i}} \set{g(\bm x+U_i(z-x_i)) +\tfrac{1}{\gamma_i}\D(\bm x+U_i(z-x_i),\bm x)}>-\infty}.
	\end{equation}
	\end{defin}
	
	\begin{prop}[characteristics of block prox-boundedness] 
	\label{pro:proxBoundedness}%
		For \(\func{h_i}{\R^n}{\Rinf}\) and proper and lsc functions $\func{g_i}{\R^{n_i}}{\Rinf}$ ($i=1,\ldots,N$), the following statements are equivalent:
		\begin{enumerateq}
		\item\label{pro:proxBoundedness1}
			$g=\sum_{i=1}^N g_i$ is block prox-bounded;
		\item\label{pro:proxBoundedness2}
			for all $i=1,\ldots,N$, $g_i+r_i h_i(\bm x+U_i(z-x_i))$ is bounded below on $\R^{n_i}$ for some $r_i\in\R$;
		\item\label{pro:proxBoundedness3}
			for all $i=1,\ldots,N$, $\liminf_{\|z\|\to \infty}\nicefrac{g_i(z)}{h_i(\bm x+U_i(z-x_i))}>-\infty$.
		\end{enumerateq}
	\end{prop}
	
	\begin{proof}
	    The proof is a straightforward adaptation of \cite[Proposition 2.7]{ahookhosh2019multi}.
	\end{proof}
	
	For a given points $\bm x^k, \bm x^{k-1}\in\R^n$ and $\alpha_i^k\geq 0$, let us define the function $\func{\M}{\dom h_i\times\interior\dom h_i\times\interior\dom h_i}{\Rinf}$ given by
	\begin{equation}\label{eq:modelM}
	   \M(\bm x,\bm x^k,\bm x^{k-1}):= \innprod{\nabla f(\bm x^k)-\tfrac{\alpha_i^k}{\gamma_i^k}(\bm x^k-\bm x^{k-1})}{\bm x-\bm x^k}+\tfrac{1}{\gamma_i^k}\D(\bm x,\bm x^k)+\sum_{i=1}^N g_i(x_i)
	\end{equation}
	and the \DEF{block inertial Bregman proximal} mapping $\ffunc{\T}{\interior\dom h_i\times \interior\dom h_i}{\R^{n_i}}$ as
	\begin{equation}\label{eq:tx}
	    \begin{split}
	        &\T(\bm x^k,\bm x^{k-1}):= \argmin_{z\in\R^{n_i}}~ \M(\bm x^k+U_i(z-x_i^k),\bm x^k,\bm x^{k-1})\\
	        &~~~= \argmin_{z\in\R^{n_i}}~ \innprod{\nabla_i f(\bm x^k)-\tfrac{\alpha_i^k}{\gamma_i^k}(\bm x_i^k-\bm x_i^{k-1})}{z-\bm x_i^k}+\tfrac{1}{\gamma_i^k}\D(\bm x^k+U_i(z-x_i^k),\bm x^k)+ g_i(z),
	    \end{split}
	\end{equation}
	which is set-valued by nonconvexity of $g_i$ ($i=1,\ldots,N$), and it reduces to the inertial Bregman forward-backward mapping for $N=1$; cf. \cite{bot2016inertial}. 
	For a given sequence $\seq{\bm x^k}$, we introduce the following notation 
	\begin{align}\label{eq:xki}
	    \bm x^{k,i}:=(x_1^{k+1},\ldots,x_i^{k+1},x_{i+1}^k,\ldots,x_N^k),
	\end{align}
	i.e., $\bm x^{k,0}=\bm x^k$ and $\bm x^{k,N}=\bm x^{k+1}$.
	Using this notation and the mapping \eqref{eq:tx}, we next introduce the \DEF{block inertial Bregman proximal  algorithm} (\refBIBPA[]); see Algorithm~\ref{alg:bibpa}. 
	
	\begin{algorithm*}[ht] 
	\algcaption{({\bf\textsc{BIBPA}}) Block Inertial Bregman Proximal Algorithm}%
	\begin{algorithmic}[1]
	\Require{%
		
		\(\bm x^0\in\interior\dom h_1\),\ $I_n=(U_1,\ldots,U_N)\in\R^{n\times n}$ with $U_i\in\R^{n\times n_i}$ and the identity matrix $I_n$, $k=0$.%
	}%
	\While{some stopping criterion is not met}
	\State\label{state:bpalmm0Xk1}%
	    $\bm x^{k,0}=\bm x^k$;
	    \For{\texttt{$i=1,\ldots,N$}} 
	     choose $\gamma_i^k$ and $\alpha_i^k$ as \cref{cor:descentProp} and compute
	        \begin{align}\label{eq:xik1}
	            x_i^{k,i}\in \T(\bm x^{k,i-1},\bm x^{k-1}),\quad \bm x^{k,i}=\bm x^{k,i-1}+U_i(x_i^{k,i}-x_i^{k,i-1});
	        \end{align}
	    \EndFor
	\State\label{state:bpalmk1}%
		 $\bm x^{k+1}=\bm x^{k,N}$,~$k= k+1$;%
	\EndWhile
	\Ensure{%
		A vector $\bm x^k$.%
	}%
	\end{algorithmic}
	\label{alg:bibpa}
	\end{algorithm*}%
	
	In order to verify the well-definedness of the iterations generated by \refBIBPA[], we next investigate some important properties of the mapping $\T$.
	
	\begin{ass}\label{ass:T}
		For all $z\in \T(\bm x,\bm y)$ and $\gamma_i^k \in (0,\nicefrac{1}{L_i})$, $\bm x+U_i(z-x_i)\in C$ and $i=1,\ldots,N$.
	\end{ass}
	
	\begin{prop}[properties of the mapping $\T$] \label{pro:proxPro} 
	    Under \Cref{ass:basic:fgh} and \Cref{ass:T}, $\gamma_i^k \in (0,\gamma_{i,g}^{h_i})$ for $i\in [N]$, and $\bm x^k, \bm x^{k-1}\in\interior\dom h_i$, the following statements are true:
	    \begin{enumerate}
	    \item \label{pro:proxPro3} 
	    $\T(\bm x^k,\bm x^{k-1})$ is nonempty, compact, and outer semicontinuous; 
	    \item \label{pro:proxPro1} 
	    $\dom \T=\interior\dom h_i\times \interior\dom h_i$; 
	    \item \label{pro:proxPro2}
	    If $x_i^{k,i}\in \T(\bm x^{k,i-1},\bm x^{k-1})$ for $\gamma_i^k \in (0,\nicefrac{1}{L_i})$, then $\bm x^{k,i}\in \interior\dom h_i$.
	    \end{enumerate}\let\qedsymbol\relax
	\end{prop}
	
	\begin{proof}
	    The proof follows from \cite[Proposition 2.10]{ahookhosh2019multi} and \Cref{ass:T}.
	\end{proof}
	
	In the subsequent lemma, we show that the cost function $\Phi$ satisfies some necessary inequality that will be needed in the next result.
	
	\begin{lem}[cyclic inequality of the cost]
	\label{pro:proxAltIneq}
	    Let \Cref{ass:basic:fgh} and \Cref{ass:T} hold, and let $\seq{\bm x^k}$ be generated by \refBIBPA[]. If $h_i$ ($i\in [N]$) is $\sigma_i$-block strongly convex, then we have
	            \begin{equation}\label{eq:cyclicIneq}
	                \Phi(\bm x^{k+1})-\Phi(\bm x^{k})\leq \sum_{i=1}^N \left(\left(\tfrac{|\alpha_i^k|}{\sigma_i\gamma_i^k}-\tfrac{1-\gamma_i^k Li}{\gamma_i^k}\right)\D(\bm x^{k,i},\bm x^{k,i-1})
	                +\tfrac{|\alpha_i^k|}{\sigma_i\gamma_i^k}\D(\bm x^{k-1,i},\bm x^{k-1,i-1})\right).
	            \end{equation}
	\end{lem}
	
	\begin{proof}
	    For $i\in\set{1,\ldots,N}$ and $x_i^{k,i}\in \T(\bm x^{k,i-1},\bm x^{k-1})$, it holds that
	    \[
	        \innprod{\nabla_i f(\bm x^{k,i-1})-\tfrac{\alpha_i^k}{\gamma_i^k}(x_i^k- x_i^{k-1})}{x_i^{k,i}-x_i^k}
	        +\tfrac{1}{\gamma_i^k}\D(\bm x^{k,i},\bm x^{k,i-1})+ \sum_{j=1}^N g_j(\bm x_j^{k,i})\leq \sum_{j=1}^N  g_j(\bm x_j^{k,i-1}).
	    \]
	    Together with \Cref{ass:basic:f} and  \Cref{fac:relSmoothEqvi2}, this implies
	    \begin{align*}
	    \begin{array}{ll}
	    &    f(\bm x^{k,i}) \leq f(\bm x^{k,i-1})+\innprod{\nabla_i f(\bm x^{k,i-1})}{x_i^{k,i}-x_i^k}+L_i \D(\bm x^{k,i},\bm x^{k,i-1})\\
	        &\leq f(\bm x^{k,i-1})+\sum_{j=1}^N  g_j(\bm x_j^{k,i-1})-\sum_{j=1}^N g_j(\bm x_j^{k,i}) -\tfrac{1-\gamma_i^k Li}{\gamma_i^k}\D(\bm x^{k,i},\bm x^{k,i-1})+\tfrac{\alpha_i^k}{\gamma_i^k}\innprod{x_i^k- x_i^{k-1}}{x_i^{k,i}-x_i^k}\\
	        &\leq f(\bm x^{k,i-1})+\sum_{j=1}^N  g_j(\bm x_j^{k,i-1})-\sum_{j=1}^N g_j(\bm x_j^{k,i})-\tfrac{1-\gamma_i^k Li}{\gamma_i^k}\D(\bm x^{k,i},\bm x^{k,i-1})\\
	        &~~~+\tfrac{|\alpha_i^k|}{2\gamma_i^k}\left(\|x_i^k- x_i^{k-1}\|^2+\|x_i^{k,i}-x_i^k\|^2\right)\\
	        &\leq f(\bm x^{k,i-1})+\sum_{j=1}^N  g_j(\bm x_j^{k,i-1})-\sum_{j=1}^N g_j(\bm x_j^{k,i})+\left(\tfrac{|\alpha_i^k|}{\sigma_i\gamma_i^k}-\tfrac{1-\gamma_i^k Li}{\gamma_i^k}\right)\D(\bm x^{k,i},\bm x^{k,i-1})\\
	        &~~~+\tfrac{|\alpha_i^k|}{\sigma_i\gamma_i^k}\D(\bm x^{k-1,i},\bm x^{k-1,i-1}),
	        \end{array}
	    \end{align*}
	    which yields 
	    \begin{equation}\label{eq:fiOverx}
	        \Phi(\bm x^{k,i})\leq \Phi(\bm x^{k,i-1})+\left(\tfrac{|\alpha_i^k|}{\sigma_i\gamma_i^k}-\tfrac{1-\gamma_i^k Li}{\gamma_i^k}\right)\D(\bm x^{k,i},\bm x^{k,i-1})
	        +\tfrac{|\alpha_i^k|}{\sigma_i\gamma_i^k}\D(\bm x^{k-1,i},\bm x^{k-1,i-1}).
	    \end{equation}
	    Now, let us sum up both sides of \eqref{eq:fiOverx} for $i=1,\ldots,N$, i.e., 
	    \begin{align*}
	    \begin{array}{ll}
	        \Phi(\bm x^{k+1})-\Phi(\bm x^{k})&=\sum_{i=1}^N \left(\Phi(\bm x^{k,i})-\Phi(\bm x^{k,i-1})\right)\\
	        &\leq \sum_{i=1}^N \left(\left(\tfrac{|\alpha_i^k|}{\sigma_i\gamma_i^k}-\tfrac{1-\gamma_i^k Li}{\gamma_i^k}\right)\D(\bm x^{k,i},\bm x^{k,i-1})
	        +\tfrac{|\alpha_i^k|}{\sigma_i\gamma_i^k}\D(\bm x^{k-1,i},\bm x^{k-1,i-1})\right).
	        \end{array}
	    \end{align*}
	\end{proof}
	
	We notice that \Cref{pro:proxAltIneq} does not guarantee the monotonicity of the sequence $\seq{\Phi(\bm x^k)}$. For $\bm x, \bm y\in\R^n$ and $\delta_i\geq 0$, we define the \DEF{Lyapunov function} $\func{\mathcal{L}}{\R^n\times\R^n}{\Rinf}$,
	\begin{equation}\label{eq:LyapunovFunc}
	    \mathcal{L}(\bm x,\bm y):= \Phi(\bm x)+\sum_{i=1}^N \delta_i \D((x_1,\ldots,x_i,y_{i+1},\ldots,y_N),(x_1,\ldots,x_{i-1},y_{i},\ldots,y_N)),
	\end{equation}
	Note that 
	$    \mathcal{L}(\bm x^{k+1},\bm x^{k}):= \Phi(\bm x^{k+1})+\sum_{i=1}^N \delta_i \D(\bm x^{k,i},\bm x^{k,i-1}).$ 
	We denote by $\mathcal{L}^{k+1}$ and $\mathcal{L}^k$ the terms $\mathcal{L}(\bm x^{k+1},\bm x^{k})$ and $\mathcal{L}(\bm x^{k},\bm x^{k-1})$, respectively.
	We next indicate the monotonicity of  $\seq{\mathcal{L}^k}$. 

	\begin{prop}[descent property of the Lyapunov function]
	\label{cor:descentProp}
	    Let \Cref{ass:basic:fgh} and \Cref{ass:T} hold, let $\seq{\bm x^k}$ be generated by \refBIBPA[], and let $h_i$ ($i=1,\ldots,N$) be $\sigma_i$-block strongly convex. If  $\lim_{k\to\infty} \alpha_i^k = \alpha_i$ and $0<\gamma_i \leq \tfrac{\sigma_i-2|\alpha_i|}{\sigma_i L_i}$ and
	    \begin{align}\label{eq:assumAlphai}
	        |\alpha_i^k| < \tfrac{\sigma_i}{2},~~  0<\gamma_i\leq\gamma_i^k \leq \tfrac{\sigma_i-2|\alpha_i^k|}{\sigma_i L_i},~~ \tfrac{|\alpha_i^k|}{\sigma_i\gamma_i^k}\leq \delta_i\leq \tfrac{1-\gamma_i^k Li}{\gamma_i^k}-\tfrac{|\alpha_i^k|}{\sigma_i\gamma_i^k}\quad i=1,\ldots,N,
	    \end{align}
	    then, setting $a_i:=\tfrac{1-\gamma_i^k Li}{\gamma_i^k}-\tfrac{|\alpha_i^k|}{\sigma_i\gamma_i^k}-\delta_i$ and $b_i:=\delta_i-\tfrac{|\alpha_i^k|}{\sigma_i\gamma_i^k}$ for $i=1,\ldots,N$, we get 
	     \begin{equation}\label{eq:LyapunovIneq}
	            \mathcal{L}^{k+1}-\mathcal{L}^k \leq -\sum_{i=1}^N \left(a_i\D(\bm x^{k,i},\bm x^{k,i-1})
	        +b_i\D(\bm x^{k-1,i},\bm x^{k-1,i-1})\right),
	        \end{equation}
	        i.e., the sequence $\seq{\mathcal{L}^k}$ is non-increasing and consequently $\lim_{k\to\infty} \D(\bm x^{k,i},\bm x^{k,i-1})=0$, i.e., $\lim_{k\to\infty} \|\bm x^{k,i}-\bm x^{k,i-1}\|=0$, for all $i=1,\ldots,N$.
	\end{prop}
	
	\begin{proof}
	Using \eqref{eq:cyclicIneq} and applying the Lyapunov function \eqref{eq:LyapunovFunc}, we have
	    \begin{align*}
	    \begin{array}{ll}
	        \mathcal{L}^{k+1}-\mathcal{L}^k &= \Phi(\bm x^{k+1})-\Phi(\bm x^k)+\sum_{i=1}^N \delta_i \D(\bm x^{k,i},\bm x^{k,i-1})-\sum_{i=1}^N \delta_i \D(\bm x^{k-1,i},\bm x^{k-1,i-1})\\
	        &\leq \sum_{i=1}^N \left(\left(\tfrac{|\alpha_i^k|}{\sigma_i\gamma_i^k}-\tfrac{1-\gamma_i^k Li}{\gamma_i^k}+\delta_i\right)\D(\bm x^{k,i},\bm x^{k,i-1})
	        +\left(\tfrac{|\alpha_i^k|}{\sigma_i\gamma_i^k}-\delta_i\right) \D(\bm x^{k-1,i},\bm x^{k-1,i-1})\right),
	        \end{array}
	    \end{align*}
	    as claimed in \eqref{eq:LyapunovIneq}.
	    In order to guarantee the non-increasing property of the sequence $\seq{\mathcal{L}^k}$, the inequalities 
	  $        a_i=\tfrac{1-\gamma_i^k Li}{\gamma_i^k}-\tfrac{|\alpha_i^k|}{\sigma_i\gamma_i^k}-\delta_i\geq 0$, $
	        b_i=\delta_i-\tfrac{|\alpha_i^k|}{\sigma_i\gamma_i^k} \geq 0
	    $
	    should be satisfied, for $i=1,\ldots,N$, i.e.,
	   $
	        \tfrac{|\alpha_i^k|}{\sigma_i\gamma_i^k}\leq \delta_i\leq \tfrac{1-\gamma_i^k Li}{\gamma_i^k}-\tfrac{|\alpha_i^k|}{\sigma_i\gamma_i^k} \leq \tfrac{1-\gamma_i Li}{\gamma_i} \quad i=1,\ldots,N,
	 $
	    which is guaranteed by \eqref{eq:assumAlphai}, i.e., $\mathcal{L}^{k+1}\leq \mathcal{L}^k$. Together with \eqref{eq:LyapunovIneq}, this yields that
	    \begin{align*}
	    \begin{array}{ll}
	        \sum_{k=0}^p \sum_{i=1}^N a_i \D(\bm x^{k,i},\bm x^{k,i-1})
	        +b_i\D(\bm x^{k-1,i},\bm x^{k-1,i-1})
	        &\leq \sum_{k=0}^p \left(\mathcal{L}^k-\mathcal{L}^{k+1}\right) \\
	        &= \mathcal{L}^0-\mathcal{L}^{p+1} \leq \mathcal{L}^0-\inf \mathcal{L}< +\infty.
	        \end{array}
	    \end{align*}
	Let $p\to +\infty$, the result follows from $\D(\cdot,\cdot)\geq 0$ and block strong convexity of $h_i$.
	\end{proof}
	
	In convergence analysis of proximal algorithms, one usual assumption is the boundedness of $\seq{\bm x^k}$; cf.,  \cite{attouch2010proximal,bolte2018first}. A sufficient condition for this is given next. 
		
	\begin{cor}[boundedness of iterations]
	\label{cor:iterComplexity}
	    Suppose that all assumptions of \Cref{cor:descentProp} hold. Further, if $\varphi$  has bounded level sets, then the sequence $\seq{\bm x^k}$ is bounded.
	\end{cor}
		
	\begin{proof}
	    It follows from \Cref{cor:descentProp} that $\mathcal{L}(\bm x^{k+1}, \bm x^k)$ is non-increasing, hence
	    \begin{align*}
	    \begin{array}{ll}
	        \Phi(\bm x^{k+1})&\leq \mathcal{L}(\bm x^{k+1},\bm x^{k})= \Phi(\bm x^{k+1})+\sum_{i=1}^N \delta_i \D(\bm x^{k,i},\bm x^{k,i-1})\\
	        &\leq \mathcal{L}(\bm x^{1},\bm x^{0})= \Phi(\bm x^{1})+\sum_{i=1}^N \delta_i \D(\bm x^{0,i},\bm x^{0,i-1})< \infty.
	        \end{array}
	    \end{align*}
	    Hence, $
	    \mathcal{N}(\bm x^1,\bm x^0):=\set{\bm x\in\R^n}[\Phi(\bm x) \leq \Phi(\bm x^{1})+\sum_{i=1}^N \delta_i \D(\bm x^{0,i},\bm x^{0,i-1})]
	    $
	     encompasses $\seq{\bm x^k}$, i.e., $\seq{\bm x^k}\subseteq \mathcal{N}(\bm x^1,\bm x^0)$. Since $\varphi$ has bounded level sets, we have  $\seq{\bm x^k}$ are bounded. 
	\end{proof}
	
	\noindent The next proposition provides a lower bound for
	$
	\sum_{i=1}^{N}\sqrt{\D_{h}(\bm x^{k,i},\bm x^{k,i-1})}+\sqrt{\D_{h}(\bm x^{k-1,i},\bm x^{k-1,i-1})}.
	$
	\begin{prop}[subgradient lower bound for iterations gap]
	\label{pro:subgradLowBound1}
	    Let  \Cref{ass:basic:fgh} and \Cref{ass:T} hold, let $\seq{\bm x^k}$ be generated by \refBIBPA[], and let $h_i$ ($i\in [N]$) be $\sigma_i$-block strongly convex. Furthermore, suppose that $\nabla_i f$, $\nabla_i h$, ($i=1,\ldots,N$) are locally Lipschitz on bounded sets with Lipschitz moduli $\widehat L$ and $\widetilde L_i>0$, $\nabla_{ii}^2 h_i$ is bounded on bounded set with constants $\overline{L}_i$ ($i\in [N]$) and that the sequence $\seq{\bm x^k}$ is bounded. For a fixed $k\in\N$ and $j\in [N]$, we define
	     \begin{equation}\label{eq:pxkp1}
	        \mathcal{G}_j^{k+1}:= (\mathcal{V}_j^{k+1}, \mathcal{W}_j^{k+1}),
	    \end{equation}
	    where
	    \begin{align*}
	    \begin{array}{ll}
	        \mathcal{V}_j^{k+1}&:= \sum\limits_{i=j}^N \delta_i (\nabla_j h_i(\bm x^{k,i})-\nabla_j h_i(\bm x^{k,i-1}))+\tfrac{1}{\gamma_j^k} (\nabla_j h_j(\bm x^{k,j-1})-\nabla_j h_j(\bm x^{k,j})) \\
	        &~~~+\tfrac{\alpha_j^k}{\gamma_j^k}(x_j^k-x_j^{k-1}) 
	        +\nabla_j f(\bm x^{k+1})-\nabla_j f(\bm x^{k,j-1})\\
	        \mathcal{W}_j^{k+1}&:= \sum\limits_{i=1}^{j-1} \delta_i (\nabla_j h_i(\bm x^{k,i})-\nabla_j h_i(\bm x^{k,i-1}))-\nabla_{jj}^2 h_j(\bm x^{k,j-1})(x_j^{k+1}-x_j^{k}).
	        \end{array}
	    \end{align*}
	If $h_i$, $i\in [N]$, is block strongly convex, then $\mathcal{G}^{k+1}:=\left(\mathcal{G}_1^{k+1},\ldots,\mathcal{G}_N^{k+1}\right)\in \partial \mathcal{L}(\bm x^{k+1},\bm x^{k})$ and
	    \begin{equation}\label{eq:subGradUpBound}
	        \|\mathcal{G}^{k+1}\|\leq \overline{c} \sum_{i=1}^{N}\sqrt{\D_{h}(\bm x^{k,i},\bm x^{k,i-1})}+\widehat{c}\sum_{i=1}^{N}\sqrt{\D_{h}(\bm x^{k-1,i},\bm x^{k-1,i-1})},
	    \end{equation}
	    with 
	    \begin{align*}
	    \begin{array}{ll}
	        &\overline{c}:= \max \set{\sqrt{\nicefrac{2}{\sigma_1}},\ldots,\sqrt{\nicefrac{2}{\sigma_N}}}\left(N\left(\widehat{L}+\max\set{\delta_1 \widetilde{L}_1,\ldots,\delta_N \widetilde{L}_N}\right)+ \max \set{\tfrac{\widetilde{L}_1}{\gamma_1}+\overline{L}_1 ,\ldots,\tfrac{\widetilde{L}_N}{\gamma_N}+\overline{L}_N }\right),\\
	        &\widehat{c}:= \max \set{\sqrt{\nicefrac{2}{\sigma_1}},\ldots,\sqrt{\nicefrac{2}{\sigma_N}}}
	        \max \set{\tfrac{\sigma_1(1-\gamma_1 L_1)}{\gamma_1},\ldots,\tfrac{\sigma_N(1-\gamma_N L_N)}{\gamma_N}}.
	        \end{array}
	    \end{align*}
	\end{prop}
	
	\begin{proof}
		Following \cite[Chapter 10]{rockafellar2011variational}, the subdifferential of $\mathcal{L}$ at $(x^{k+1},x^k)$ is given by 
		\begin{equation}\label{eq:subDiffLyap}
		\partial \mathcal{L}(\bm x^{k+1}, \bm x^k)=\left(\partial_{\bm x^{k+1}} \mathcal{L}(\bm x^{k+1},\bm x^k), \partial_{\bm x^{k}} \mathcal{L}(\bm x^{k+1}, \bm x^k)\right), 
		\end{equation}
		where, for $j=1,\ldots,N$ , by applying \cite[Exercise 8.8]{rockafellar2011variational} we have 
		\begin{align}
			& \partial_{\bm x_j^{k+1}} \mathcal{L}(\bm x^{k+1},\bm x^k)=\nabla_j f(\bm x^{k+1})+\partial g_j(\bm x_j^{k+1})+\sum_{i=j}^N \delta_i (\nabla_j h_i(\bm x^{k,i})-\nabla_j h_i(\bm x^{k,i-1}));\label{eq:subDiffLyapPar1}\\
			& \partial_{\bm x_j^{k}} \mathcal{L}(\bm x^{k+1},\bm x^k)=\sum_{i=1}^{j-1} \delta_i (\nabla_j h_i(\bm x^{k,i})-\nabla_j h_i(\bm x^{k,i-1}))-\nabla_{jj}^2 h_j(\bm x^{k,j-1})(x_j^{k+1}-x_j^{k}). \label{eq:subDiffLyapPar2}
		\end{align}
	    Writing the first-order optimality conditions for the subproblem \eqref{eq:tx} implies that there exists a subgradient $\eta_j^{k+1}\in\partial g_j(x_j^{k+1})$ such that
	    \[
	        \nabla_j f(\bm x^{k,j-1})-\tfrac{\alpha_j^k}{\gamma_j^k}(x_j^k-x_j^{k-1})+\tfrac{1}{\gamma_j^k} \left(\nabla_j h_j(\bm x^{k,j})-\nabla_j h_j(\bm x^{k,j-1})\right)+\eta_j^{k+1}=0 \quad j\in [N],
	    \]
	    which implies $ \eta_j^{k+1}=\tfrac{1}{\gamma_j^k} \left(\nabla_j h_j(\bm x^{k,j-1})-\nabla_j h_j(\bm x^{k,j})\right)+\tfrac{\alpha_j^k}{\gamma_j^k}(x_j^k-x_j^{k-1})-\nabla_j f(\bm x^{k,j-1})$, $ j \in [N]$.    
	    	Therefore, we have $
		 	\mathcal{V}_j^{k+1}=\nabla_j f(\bm x^{k+1})+\eta_j^{k+1}+\sum_{i=j}^N \delta_i (\nabla_j h_i(\bm x^{k,i})-\nabla_j h_i(\bm x^{k,i-1}))\in \partial_{\bm x_j^{k+1}} \mathcal{L}(\bm x^{k+1},\bm x^k),$
	    which implies $\mathcal{G}^{k+1}\in \partial \mathcal{L}(\bm x^{k+1}, \bm x^k)$. 
	    Together with the Lipschitz continuity of $\nabla_i f$, $\nabla_i h_i$ and the boundedness of $\nabla_{ii}^2 h_i$ on bounded sets, the boundedness of $\seq{\bm x^k}$, and the triangle inequality, this implies that there exist constants $\widehat L,~ \widehat L_i,~ \overline{L}_i>0$ (for $i\in [N]$) such that 
	     \begin{align*}
	    &    \|\mathcal{G}_j^{k+1}\|=\|\mathcal{V}_j^{k+1}\|+\|\mathcal{W}_j^{k+1}\|\leq
	        \tfrac{\alpha_j^k}{\gamma_j^k} \|x_j^k-x_j^{k-1}\| 
	        +\|\nabla_j f(\bm x^{k+1})-\nabla_j f(\bm x^{k,j-1})\| \\
	        &~~~+\sum_{i=1}^N \delta_i \|\nabla_j h_i(\bm x^{k,i})-\nabla_j h_i(\bm x^{k,i-1})\|+\tfrac{1}{\gamma_j^k} \|\nabla_j h_j(\bm x^{k})-\nabla_j h_j(\bm x^{k,1})\|+\|\nabla_{jj}^2 h_j(\bm x^{k,j-1})\| ~\|x_j^{k+1}-x_j^{k}\|\\
	        &\leq \tfrac{\alpha_j^k}{\gamma_j^k}\|x_j^k-x_j^{k-1}\|+ \widehat L \sum_{i=1}^{N}\|x_i^{k+1}-x_i^k\|+ \sum_{i=1}^N \delta_i \widetilde{L}_i \|x_i^{k+1}-x_i^k\|+\left(\tfrac{\widetilde{L}_j}{\gamma_j^k}+\overline{L}_i\right) \|x_j^{k+1}-x_j^k\|.
	    \end{align*}
	    Combining the last two inequalities with  \eqref{eq:assumAlphai}, it can be deduced that
	    \begin{align*}
	    \begin{array}{ll}
	        \|\mathcal{G}^{k+1}\|
	       & \leq\left(N\left(\widehat{L}+\max\set{\delta_1 \widetilde{L}_1,\ldots,\delta_N \widetilde{L}_N}\right)+\max \set{\tfrac{\widetilde{L}_1}{\gamma_1^k}+\overline{L}_1 ,\ldots,\tfrac{\widetilde{L}_N}{\gamma_N^k}+\overline{L}_N }\right) \sum_{i=1}^{N}\|x_i^{k+1}-x_i^k\|\\
	       &~~~
	        +\max \set{\tfrac{\alpha_1^k}{\gamma_1^k},\ldots,\tfrac{\alpha_N^k}{\gamma_N^k}} \sum_{i=1}^{N}\|x_i^k-x_i^{k-1}\|\\
	        & \leq\left(N\left(\widehat{L}+\max\set{\delta_1 \widetilde{L}_1,\ldots,\delta_N \widetilde{L}_N}\right)+ \max \set{\tfrac{\widetilde{L}_1}{\gamma_1}+\overline{L}_1 ,\ldots,\tfrac{\widetilde{L}_N}{\gamma_N}+\overline{L}_N }\right) \sum_{i=1}^{N}\|x_i^{k+1}-x_i^k\|\\
	        &~~~+\max \set{\tfrac{\sigma_1(1-\gamma_1 L_1)}{\gamma_1},\ldots,\tfrac{\sigma_N(1-\gamma_N L_N)}{\gamma_N}} \sum_{i=1}^{N}\|x_i^k-x_i^{k-1}\|\\
	        &\leq \overline{c} \sum_{i=1}^{N}\|x_i^{k+1}-x_i^k\|+\widehat{c}\sum_{i=1}^{N}\|x_i^k-x_i^{k-1}\|.
	        \end{array}
	    \end{align*}
	    Hence, it follows from the block strong convexity of $h_i$ ($i=1,\ldots,N$) that
	    \begin{align*}
	    \begin{array}{ll}
	        \|\mathcal{G}^{k+1}\|
	        &\leq \overline{c} \sum_{i=1}^{N}\|x_i^{k+1}-x_i^k\|+\widehat{c}\sum_{i=1}^{N}\|x_i^k-x_i^{k-1}\|\\
	        &\leq \overline{c} \sum_{i=1}^{N}\sqrt{\D_{h}(\bm x^{k,i},\bm x^{k,i-1})}+\widehat{c}\sum_{i=1}^{N}\sqrt{\D_{h}(\bm x^{k-1,i},\bm x^{k-1,i-1})},
	        \end{array}
	    \end{align*}
	    giving our desired result.
	\end{proof}
	
	\vspace{-6mm}
	\begin{rem}
	Note that a uniformly continuous function maps bounded sets to bounded sets. Therefore, in \Cref{pro:subgradLowBound1}, if the function $\nabla_{ii}^2 h_i$ ($i=1,\ldots,N$) is uniformly continuous, it is bounded on bounded sets.
	\end{rem}
	
	Applying \Cref{pro:subgradLowBound1}, the \DEF{subsequential convergence} of $\seq{\bm x^k}$ generated by \refBIBPA[] is presented next. On top of that we explain some basic properties of $\omega(\bm x^0)$.

	\begin{ass}
	$\overline{C}\subseteq \interior\dom h_1$.
	\end{ass}
	\begin{thm}[subsequential convergence and properties of $\omega(\bm x^0)$]
	\label{pro:chatClusterPoint10}
	 Let all assumptions of \Cref{pro:subgradLowBound1} and \Cref{ass:T} hold. Then, the following assertions are satisfied:
	    \begin{enumerate}
	        \item \label{pro:chatClusterPoint11}
	            every cluster point of $\seq{\bm x^k}$ is a critical point of $\Phi$, i.e., $ \omega(\bm x^0)\subset \mathbf{crit}~ \Phi$;
	        \item \label{pro:chatClusterPoint12}
	            $\lim_{k\to\infty}\dist\left(\bm x^k,\omega(\bm x^0)\right)=0$;
	        \item \label{pro:chatClusterPoint13}
	        $\omega(\bm x^0)$ is a nonempty, compact, and connected set;
	        \item \label{pro:chatClusterPoint14}
	        the Lyapunov function $\mathcal{L}$ is finite and constant on $\omega(\bm x^0)$.
	    \end{enumerate}
	\end{thm}
	
	\begin{proof}
	    Let us assume $\bm x^\star=(x_1^\star,\ldots,x_N^\star)\in\omega(\bm x^0)$. The boundedness of $\seq{\bm x^k}$ implies that there exists an infinite index set $\mathcal{J}\subset \N$ such that the subsequence $\seq{\bm x^k}[k\in\mathcal{J}]\to \bm x^\star$ as $k\to\infty$. 
	    It follows from \eqref{eq:xik1} that
	    \begin{equation}\label{eq:inprodIneq}
	        \begin{array}{ll}
	            \innprod{\nabla_i f(\bm x^{k,i-1})-\tfrac{\alpha_i^k}{\gamma_i^k}(x_i^k- x_i^{k-1})}{x_i^{k+1}-x_i^k}
	        +\tfrac{1}{\gamma_i^k}\D(\bm x^{k,i},\bm x^{k,i-1})+g_i(x_i^{k+1})\\
	        \leq \innprod{\nabla_i f^k(\bm x^{k,i-1})-\tfrac{\alpha_i^k}{\gamma_i^k}(x_i^k- x_i^{k-1})}{x_i^\star-x_i^k}
	        +\tfrac{1}{\gamma_i^k}\D(\bm x^\star,\bm x^{k,i-1})+g_i(x_i^\star).
	        \end{array}
	    \end{equation}
	    Invoking \Cref{cor:descentProp} and using block strong convexity of $h_i$, there exist $\varepsilon_i^\star>0$, $k_i^0\in\N$, and a neighborhood $\mathbf{B}(x_i^\star,\varepsilon_i^\star)$ such that 
	    $
	        \lim_{k\to\infty} \tfrac{\sigma_i}{2} \|x_i^{k+1}-x_i^k\|^2 \leq\lim_{k\to\infty} \D_{h_i}(\bm x^{k,i},\bm x^{k,i-1})=0, \quad x_i^k\in \mathbf{B}(x_i^\star$, $\varepsilon_i^\star),~ i\in [N],
	   $
	for $k\geq k_i^0$ and $k\in\mathcal{J}$, i.e., $\lim_{k\to\infty} (x_i^{k+1}-x_i^k)=0$.  Hence, substituting $k=k_j-1$ for $k_j\in\mathcal{J}$ into \eqref{eq:inprodIneq} and taking the limit from both sides of \eqref{eq:inprodIneq}, we derive
	   $
	        \limsup_{j\to\infty} g_i(x_i^{k_j})\leq g_i(x_i^\star)\quad  i=1 \in [N].
	  $
	Furthermore, since $g_i$ is lsc, this yields that $\lim_{j\to\infty} g_i(x_i^{k_j})= g_i(x_i^\star)$, then
	    \begin{align*}
	        \lim_{j\to\infty} \mathcal{L}(\bm x^{k_j+1},\bm x^{k_j}) &= \lim_{j\to\infty} \left(f(x_1^{k_j},\ldots,x_N^{k_j})+\sum_{i=1}^N g_i(x_i^{k_j})+\sum_{i=1}^N \delta_i \D(\bm x^{k_j,i},\bm x^{k_j,i-1})\right)= \mathcal{L}(\bm x^\star,\bm x^\star).
	    \end{align*}
	    Hence, from \eqref{eq:subGradUpBound} and \Cref{cor:descentProp}, we obtain
	    \begin{align*}
	        \lim_{k\to+\infty} \|\mathcal{G}^{k+1}\|&\leq \lim_{k\to+\infty} \left(\overline{c} \sum_{i=1}^{N} \sqrt{\D_{h_i}(\bm x^{k,i},\bm x^{k,i-1})}+\widehat{c}\sum_{i=1}^{N}\sqrt{\D_{h_i}(\bm x^{k-1,i},\bm x^{k-1,i-1})}\right) = 0,
	    \end{align*} 
	    which consequently yields $\lim_{k\to\infty} \mathcal{G}^{k+1}=0$. As a result, we have $0\in\partial \mathcal{L} (x^\star,x^\star)$, owing to the closedness of the subdifferential mapping $\partial \mathcal{L}$. The result of \Cref{pro:chatClusterPoint11} follows from the fact $ \partial \mathcal{L}(\bm x^\star,\bm x^\star)=\left(\partial \Phi(\bm x^\star), 0\right)$.
	 Moreover, \Cref{pro:chatClusterPoint12} is a straightforward consequence of \Cref{pro:chatClusterPoint11}, and \Cref{pro:chatClusterPoint13} and \Cref{pro:chatClusterPoint14} can be proved in the same way as \cite[Lemma 5(iii)-(iv)]{bolte2014proximal}.
	\end{proof}

	\subsection{Global convergence for K{{\L}} functions}
	In this section, we consider the class of Kurdyka-{\L}ojasiewicz (K{{\L}}) functions (see \cite{kurdyka1998gradients,lojasiewicz1993geometrie}) and show that for such functions the sequence $\seq{\bm x^k}$  converges to a critical point $x^\star$.
	
	\begin{defin}[K{{\L}} property]\label{def:klFunctions1}
	A proper and lsc function $\func{\varphi}{\R^{n}}{\Rinf}$ has the K{{\L}} property 
	 at $\bm x^\star\in\dom \varphi$ if there exist a concave function 
	$\psi:[0,\eta]\to{[0,+\infty[}$ (with $\eta>0$) and neighborhood $\ball{\bm x^\star}{\varepsilon}$ with $\varepsilon>0$, such that
	 (i) $\psi(0)=0$;
	(ii) $\psi$ is of class $\mathcal{C}^1$ with $\psi>0$ on $(0,\eta)$;
	(ii) for all $\bm x\in \ball{\bm x^\star}{\varepsilon}$ such that $\varphi(\bm x^\star)<\varphi(\bm x)<\varphi(\bm x^\star)+\eta$ it holds that
	\begin{equation}\label{eq:klProperty1}
	    \psi'(\varphi(\bm x)-\varphi(\bm x^\star))\dist(0,\partial \varphi(\bm x))\geq 1.
	\end{equation}
	If this property holds for each point of $\dom \partial \varphi$, the $\varphi$ is a \DEF{K{{\L}} function}.
	\end{defin}
	
	In \cite{lojasiewicz1963propriete,lojasiewicz1993geometrie}, Stanis{\l}aw {\L}ojasiewicz showed for the first time that every real analytic function\footnote{A function $\func{\varphi}{\R^n}{\Rinf}$ said to be real analytic if it can be represented by a convergent power series.} satisfies \eqref{eq:klProperty1} with $\psi(s):=\frac{\kappa}{1-\theta} s^{1-\theta}$ with $\theta\in [0,1)$. 
	In 1998,  Kurdyka \cite{kurdyka1998gradients} proved that this inequality is valid for $\mathcal{C}^1$ functions whose graph belong to an \DEF{$o$-minimal structure} (see its definition in \cite{van1998tame}). Later,  \eqref{eq:klProperty1} was extended for nonsmooth functions in \cite{bolte2007clarke,bolte2007lojasiewicz,bolte2010characterizations}.
	
	The K{{\L}} property \eqref{eq:klProperty1} of the underlying objective function plays a key role in establishing the global convergence of a generic algorithm for nonconvex problems; however, this is not sufficient and one also needs some additional conditions to be guaranteed by the algorithm (see below). 
	 In particular, for several algorithms the cost functions satisfy the sufficient decrease condition (cf. \cite{ahookhosh2019multi,attouch2013convergence,bolte2014proximal}), while for some others the sufficient decrease condition is satisfied for some Lyapunov functions (cf. \cite{frankel2015splitting,ochs2014ipiano,ochs2019unifying,pock2016inertial,zhang2019inertial}). 
	
	As shown in \Cref{cor:descentProp}, \Cref{pro:subgradLowBound1}, and \Cref{pro:chatClusterPoint10} (see its proof), the sequence $\seq{\bm x^k}$ generated by \refBIBPA[] satisfies the following conditions that are non-Euclidean extension of those given in \cite{attouch2013convergence,bolte2014proximal} for the structured problem \eqref{eq:P}:
	\begin{enumerate}
	    \item \label{ass:suffCond}(\DEF{sufficient descent condition})
	    For each $k\in\N$ and $a_i, b_i\geq 0$ ($i=1,\ldots,N$),
	    \begin{align*}
	         \sum_{i=1}^N \left(a_i\D(\bm x^{k,i},\bm x^{k,i-1})
	        +b_i \D(\bm x^{k-1,i},\bm x^{k-1,i-1})\right)\leq \mathcal{L}(\bm x^k,\bm x^{k-1})-\mathcal{L}(\bm x^{k+1},\bm x^k);
	    \end{align*}
	    \item \label{ass:subgradIneq} (\DEF{subgradient lower bound of iteration gap})
	    For each $k\in\N$, there exists a subgradient $\mathcal{G}^{k+1}\in \partial \mathcal{L}(\bm x^{k+1},\bm x^k)$ and $\overline{c}, \widehat{d}\geq 0$ such that
	    \begin{align*}
	         \|\mathcal{G}^{k+1}\|\leq  \overline{c}\sum_{i=1}^{N} \sqrt{\D(\bm x^{k,i},\bm x^{k,i-1})}+ \widehat{c}\sum_{i=1}^{N} \sqrt{\D(\bm x^{k-1,i},\bm x^{k-1,i-1})};
	    \end{align*}
	    \item \label{ass:continuity} (\DEF{continuity condition})
	    The function $\mathcal{L}$ is a K{{\L}} function, and each cluster point $\bm x^\star$ of $\seq{\bm x^k}$ ($\bm x^\star\in\omega(\bm x^0)$) satisfies $(\bm x^\star,\bm x^\star) \in \mathrm{crit}\mathcal{L}$ 
	\end{enumerate}
	We now use the above three conditions to prove that the whole sequence $\seq{\bm x^k}$ converges.
	\begin{thm}[global convergence]
	\label{thm:globConvergence}
	    Let all assumptions of \Cref{pro:subgradLowBound1} and \Cref{ass:T} hold. If $\mathcal{L}$ is a K{{\L}} function, then the following statements are true:
	    \begin{enumerate}
	        \item \label{thm:globConvergence1}
	        The sequence $\seq{\bm x^k}$ has finite length, i.e.,
	        \begin{equation}\label{eq:finiteLength1}
	            \sum_{k=1}^\infty \|x_i^{k+1}-x_i^k\|<\infty \quad i=1,\ldots,N;
	        \end{equation}
	        \item \label{thm:globConvergence2} 
	        The sequence $\seq{\bm x^k}$ converges to a stationary point $\bm x^\star$ of $\Phi$.
	    \end{enumerate}
	\end{thm}
	
	\begin{proof}
		Define the sequence $\seq{d_k}$ as
	   $
	    		d_k:=\sum_{i=1}^{N} \sqrt{\D(\bm x^{k,i},\bm x^{k,i-1})}+\sqrt{\D(\bm x^{k-1,i},\bm x^{k-1,i-1})} .
	    $
	    From \Cref{pro:subgradLowBound1} for $\widetilde{c}:=\max\set{\overline{c},\widehat{c}}$, we obtain
	    \begin{equation}\label{eq:subGradIequality0}
	    	\begin{array}{ll}
	    		\|\mathcal{G}^{k+1}\|&\leq \overline{c} \sum_{i=1}^{N}\sqrt{\D_{h}(\bm x^{k,i},\bm x^{k,i-1})}+\widehat{c}\sum_{i=1}^{N}\sqrt{\D_{h}(\bm x^{k-1,i},\bm x^{k-1,i-1})}\\
	    	&\leq \widetilde{c} \sum_{i=1}^{N} \left(\sqrt{\D(\bm x^{k,i},\bm x^{k,i-1})}+\sqrt{\D(\bm x^{k-1,i},\bm x^{k-1,i-1})}\right) = \widetilde{c} d_k.
	    		\end{array}  
	    \end{equation}
		Applying twice the root-mean square and arithmetic mean inequality
		we come to
	     \begin{equation}\label{eq:rmsamiInequality}
	     	\begin{array}{ll}
	        d_k 
	            &\leq \sqrt{N\sum_{i=1}^{N} \D(\bm x^{k,i},\bm x^{k,i-1})}+\sqrt{N\sum_{i=1}^{N} \D(\bm x^{k-1,i},\bm x^{k-1,i-1})}\\
	            &\leq \sqrt{2N\sum_{i=1}^{N} \left(\D(\bm x^{k,i},\bm x^{k,i-1})+\D(\bm x^{k-1,i},\bm x^{k-1,i-1})\right)}.
	          \end{array}  
	    \end{equation}
	    Then, it can be concluded from \Cref{cor:descentProp} and \eqref{eq:rmsamiInequality} that
	     \begin{align*}
	     \begin{array}{ll}
	         \mathcal{L}^k-\mathcal{L}^{k+1}&\geq 
	         \sum_{i=1}^N \left(a_i\D(\bm x^{k,i},\bm x^{k,i-1})+b_i \D(\bm x^{k-1,i},\bm x^{k-1,i-1})\right)\\
	         &\geq 
	         \varrho \sum_{i=1}^N \left(\D(\bm x^{k,i},\bm x^{k,i-1})+ \D(\bm x^{k-1,i},\bm x^{k-1,i-1})\right)\geq \tfrac{\varrho}{2N} d_k^2,
	         \end{array}
	    \end{align*}
	    where $\varrho:=\min\set{a_1,b_1,\ldots,a_N,b_N}$. Together with \eqref{eq:subGradIequality0} and \Cref{pro:chatClusterPoint11}, this implies that \cite[Assumption H]{ochs2019unifying} holds true with $a_k=\tfrac{\varrho}{2N}, b_k=1, b=\widetilde{c}, I=\set{1}, \varepsilon_k=0$. Therefore, since $\mathcal{L}$ is a proper lower semicontinuous K{{\L}} function, \cite[Theorem 10]{ochs2019unifying} yields that \Cref{thm:globConvergence1} holds true and the sequence $\seq{\bm x^k}$ converges to $\bm x^\star$ in which $(\bm x^\star,\bm x^\star)$ is a stationary point of the Lyapunov function $\mathcal{L}$ \eqref{eq:LyapunovFunc}, i.e., $0\in \partial \mathcal{L}(\bm x^\star,\bm x^\star)$. Finally, the result follows from the fact $\partial \mathcal{L}(\bm x^\star,\bm x^\star)=\left(\partial \Phi(\bm x^\star), 0\right)$.
	\end{proof}
	\vspace{-5mm}
	\subsection{Rate of convergence for {\L}ojasiewicz-type K{{\L}} functions}
	 We now investigate the convergence rate of the generated sequence  under K{{\L}} inequality of {\L}ojasiewicz-type at $x^\star$ ($\psi(s):=\frac{\kappa}{1-\theta} s^{1-\theta}$ with $\theta\in [0,1)$), i.e., there exists $\varepsilon>0$ such that
	 \begin{equation}\label{eq:LojaKL1}
	     |\varphi(\bm x)-\varphi^\star|^{\theta}\leq \kappa \dist(0,\partial \varphi(\bm x)) \quad \forall \bm x\in \ball{\bm x^\star}{\varepsilon}.
	 \end{equation}
	
	\begin{fact}[convergence rate of a sequence with positive elements]
		\label{fac:convRate1} \cite[Lemma 15]{bot2018proximal}
		    Let $\seq{s_k}$ be a monotonically decreasing sequence in $\R_+$ and let $\theta\in [0,1)$ and $\beta>0$. Suppose that $s_k^{2\theta}\leq\beta(s_k-s_{k+1})$ holds for all $k\in\N$.
		    Then, the following assertions hold:
		    \begin{enumerate}
		        \item If $\theta=0$, the sequences $\seq{s_k}$ converges in a finite time;
		        \item If $\theta\in(0,\nicefrac{1}{2}]$, there exist $\lambda>0$ and $\tau\in[0,1)$ such that $
		            0\leq s_k\leq \lambda \tau^k
		       $ for every $k\in\N$. 
		       
		        \item If $\theta\in(\nicefrac{1}{2},1)$, there exists $\mu>0$ such that $0\leq s_k\leq \mu k^{-\tfrac{1}{2\theta-1}}$ for every $k\in\N$
		       
		    \end{enumerate}
	\end{fact}
	
	Let $\seq{\mathcal{S}_k}$ given by $\mathcal{S}_k:=\mathcal{L}(\bm x^k,\bm x^{k-1})-\mathcal{L}(\bm x^\star,\bm x^\star)$. We next derive the \DEF{convergence rates} of $\seq{\bm x^k}$ and $\seq{\mathcal{S}_k}$ when $\mathcal{L}$ satisfies the K{{\L}} inequality of {\L}ojasiewicz type. 
	
	 \begin{thm}[convergence rate]
	 \label{thm:convRate1}
	     Let all assumptions of \Cref{pro:subgradLowBound1} and \Cref{ass:T} hold, and  $\seq{\bm x^k}$ converges to $\bm x^\star$. If $\mathcal{L}$ satisfies the K{{\L}} inequality of {\L}ojasiewicz type, then the following assertions hold:
	     \begin{enumerate}
	         \item if $\theta=0$, then the sequences $\seq{\bm x^k}$ and $\seq{\Phi(\bm x^k)}$ converge in a finite number of steps to $\bm x^\star$ and $\Phi(\bm x^\star)$, respectively;
	         \item if $\theta\in(0,\nicefrac{1}{2}]$, then there exist $\lambda_1>0$, $\mu_1>0$, $\tau, \overline \tau\in [0,1)$, and $\overline{k}\in\N$ such that 
	         \[
	             0\leq \|\bm x^k-\bm x^\star\|\leq \lambda_1 \tau^k, \quad 0\leq \mathcal{S}_k\leq \mu_1 \overline\tau^k\quad \forall k\geq \overline{k};
	         \]
	         \item if $\theta\in(\nicefrac{1}{2},1)$, then there exist $\lambda_2>0$, $\mu_2>0$, and $\overline{k}\in\N$ such that
	         \[
	             0\leq \|\bm x^k-\bm x^\star\|\leq \lambda_2 k^{-\tfrac{1-\theta}{2\theta-1}}, \quad 0\leq \mathcal{S}_k\leq \mu_2 k^{-\tfrac{1-\theta}{2\theta-1}} \quad \forall k\geq \overline{k}+1.
	         \]
	     \end{enumerate}
	 \end{thm}
	
	 \begin{proof}
	     We first set $\varepsilon>0$ to be that a constant described in \eqref{eq:LojaKL1} and $x^k\in \ball{x^\star}{\varepsilon}$ for all $k\geq\tilde k$ and $\tilde k\in\N$. Let us define $\Delta_k:=\psi(\mathcal{L}(\bm x^k,\bm x^{k-1})-\mathcal{L}(\bm x^\star,\bm x^\star))=\psi(\mathcal{S}_k)$. Then, it follows from the concavity of $\psi$ and \Cref{ass:subgradIneq} that
	    \begin{align*}
	    \begin{array}{ll}
	        \Delta_k-\Delta_{k+1}&= \psi(\mathcal{S}_{k})-\psi(\mathcal{S}_{k+1})
	        \geq \psi'(\mathcal{S}_{k})(\mathcal{S}_{k}-\mathcal{S}_{k+1})\\
	        &=\psi'(\mathcal{S}_{k})(\mathcal{L}(\bm x^k, \bm x^{k-1})-\mathcal{L}(\bm x^{k+1},\bm x^k))
	        \geq \frac{\mathcal{L}(\bm x^k, \bm x^{k-1})-\mathcal{L}(\bm x^{k+1},\bm x^k)}{\dist(0,\partial \mathcal{L}(\bm x^k,\bm x^{k-1}))}\\
	        &\geq \frac{\sum_{i=1}^N \left(a_i\D(\bm x^{k,i},\bm x^{k,i-1})
	        +b_i \D(\bm x^{k-1,i},\bm x^{k-1,i-1})\right)}{\overline{c}\sum_{i=1}^{N} \sqrt{\D(\bm x^{k-1,i},\bm x^{k-1,i-1})}+ \widehat{c}\sum_{i=1}^{N} \sqrt{\D(\bm x^{k-2,i},\bm x^{k-2,i-1})}}\\
	        &\geq \frac{1}{c} ~ \frac{\sum_{i=1}^N \left(\D(\bm x^{k,i},\bm x^{k,i-1})
	        + \D(\bm x^{k-1,i},\bm x^{k-1,i-1})\right)}{\sum_{i=1}^{N} \left(\sqrt{\D(\bm x^{k-1,i},\bm x^{k-1,i-1})}+\sqrt{\D(\bm x^{k-2,i},\bm x^{k-2,i-1})}\right)},
	        \end{array}
	    \end{align*}
	with $c:=\nicefrac{\max\set{\overline c,\widehat c}}{\min\set{a_1,b_1,\ldots,a_N,b_N}}$. Using \eqref{eq:rmsamiInequality} and applying the arithmetic mean and geometric mean inequality, 
	it can be concluded that
	    \begin{equation}\label{eq:ineqakbk1}
	        \begin{array}{ll}
	           d_k
	            &\leq \sqrt{2N\sum_{i=1}^{N} \left(\D(\bm x^{k,i},\bm x^{k,i-1})+\D(\bm x^{k-1,i},\bm x^{k-1,i-1})\right)}\\
	            &\leq \sqrt{2c N\left(\Delta_k-\Delta_{k+1}\right)\sum_{i=1}^{N} \left(\sqrt{\D(\bm x^{k-1,i},\bm x^{k-1,i-1})}+\sqrt{\D(\bm x^{k-2,i},\bm x^{k-2,i-1})}\right)}\\
	            &\leq c N\left(\Delta_k-\Delta_{k+1}\right)+\tfrac{1}{2}\sum_{i=1}^{N} \left(\sqrt{\D(\bm x^{k-1,i},\bm x^{k-1,i-1})}+\sqrt{\D(\bm x^{k-2,i},\bm x^{k-2,i-1})}\right)
	        \end{array}
	    \end{equation}
	    We now define the sequences $\seq{a_k}$ and $\seq{b_k}$ given by
	    \begin{equation}\label{eq:akbk1}
	        p_{k+1}:=\sum_{i=1}^{N} \sqrt{\D(\bm x^{k,i},\bm x^{k,i-1})}+\sqrt{\D(\bm x^{k-1,i},\bm x^{k-1,i-1})},\, q_k=c N \left(\Delta_k-\Delta_{k+1}\right),\, \alpha:=\tfrac{1}{2},
	    \end{equation}
	    where
	    $
	        \sum_{i=1}^\infty q_k= 2cN \sum_{i=1}^\infty \left(\Delta_i-\Delta_{i+1}\right)
	        = \Delta_1-\Delta_\infty = \Delta_1<\infty
	    $. This and \eqref{eq:ineqakbk1} yield $p_{k+1}\leq \tfrac{1}{2} p_k+q_k$ for all $k\geq\tilde k$. Since $\seq{\Phi}$ is non-increasing, 
	     \begin{align*}
	         \sum_{j=k}^\infty p_{j+1}\leq \tfrac{1}{2} \sum_{j=k}^\infty (p_j-p_{j+1}+p_{j+1})+ 2cN\sum_{j=k}^\infty  \left(\Delta_j-\Delta_{j+1}\right)= \tfrac{1}{2}\sum_{j=k}^\infty p_{j+1}+\tfrac{1}{2} p_k+2cN \Delta_k.
	     \end{align*}
	  From the root-mean square, the arithmetic mean inequality, $\psi(\mathcal{S}_{k})\leq \psi(\mathcal{S}_{k-1})$, and \Cref{cor:descentProp}, this lead to 
	     \begin{equation}\label{eq:sumak1}
	         \begin{array}{ll}
	             \sum\limits_{j=k}^\infty p_{j+1}&\leq p_k+4cN \Delta_k = \sum\limits_{i=1}^{N} \left(\sqrt{\D(\bm x^{k-1,i},\bm x^{k-1,i-1})}+\sqrt{\D(\bm x^{k-2,i},\bm x^{k-2,i-1})}\right)+4cN \psi(\mathcal{S}_k)\\
	             &\leq \sqrt{N\sum_{i=1}^{N} \D(\bm x^{k,i},\bm x^{k,i-1})}+\sqrt{N\sum_{i=1}^{N} \D(\bm x^{k-1,i},\bm x^{k-1,i-1})}+4cN \psi(\mathcal{S}_k)\\
	            &\leq \sqrt{2N\sum_{i=1}^{N} \left(\D(\bm x^{k,i},\bm x^{k,i-1})+\D(\bm x^{k-1,i},\bm x^{k-1,i-1})\right)}+4cN \psi(\mathcal{S}_k)\\
	             &\leq \sqrt{\nicefrac{2N}{\varrho}} \sqrt{\mathcal{S}_{k-1}-\mathcal{S}_{k}}+4cN \psi(\mathcal{S}_{k-1}),
	         \end{array}
	     \end{equation}
		with $\varrho:=\min\set{a_1,b_1,\ldots,a_N,b_N}$.  
		Since $\D(\cdot,\cdot)\geq 0$, for $i=1,\ldots,N$, it holds that
	     \begin{align*}
	     \begin{array}{ll}
	       &  \|x_i^k-x_i^\star\|\leq\|x_i^{k+1}-x_i^k\|+\|x_i^{k+1}-x_i^\star\|\leq\ldots\leq\sum_{j=k}^\infty \|x_i^{j+1}-x_i^j\|\\
	         &\leq \sum_{j=k}^\infty \sqrt{\tfrac{2}{\sigma_i}\D(\bm x^{k-1,i},\bm x^{k-1,i-1})}\leq \sqrt{\tfrac{2}{\sigma_i}}\sum_{j=k}^\infty \left(\sqrt{\D(\bm x^{k-1,i},\bm x^{k-1,i-1})}+\sqrt{\D(\bm x^{k-2,i},\bm x^{k-2,i-1})}\right).
	         \end{array}
	     \end{align*}
	     Combining this with \eqref{eq:sumak1} and setting $\rho:=\max\set{\sqrt{\nicefrac{2}{\sigma_1}},\ldots,\sqrt{\nicefrac{2}{\sigma_N}}}$, we come to
	     \begin{align*}
	     \begin{array}{ll}
	     	\sum_{i=1}^N\|x_i^k-x_i^\star\|&\leq \rho\sum_{j=k}^\infty \sum_{i=1}^N\left(\sqrt{\D(\bm x^{k-1,i},\bm x^{k-1,i-1})}+\sqrt{\D(\bm x^{k-2,i},\bm x^{k-2,i-1})}\right)\\
	     	&\leq \rho \sqrt{\nicefrac{2N}{\varrho}}  \sqrt{\mathcal{S}_{k-1}-\mathcal{S}_{k}}+4c\rho N \psi(\mathcal{S}_{k-1}),
	     	\end{array}
	     \end{align*}
	     which consequently yields
	     \begin{equation}\label{eq:xkykUpperBound}
	         \|x_i^k-x_i^\star\|\leq \nu \max\set{\sqrt{\mathcal{S}_{k-1}},\psi(\mathcal{S}_{k-1})}\quad i=1,\ldots,N,
	     \end{equation}
	     with $\nu:=\rho \sqrt{\nicefrac{2N}{\varrho}}+4c\rho N$ and $\psi(s):=\frac{\kappa}{1-\theta} s^{1-\theta}$. Furthermore, the nonlinear equation
	         $\sqrt{\mathcal{S}_{k-1}}-\frac{\kappa}{1-\theta} \mathcal{S}_{k-1}^{1-\theta}=0$ has a solution at  $\mathcal{S}_{k-1}=\left(\nicefrac{(1-\theta)}{\kappa}\right)^{\tfrac{2}{1-2\theta}}$. For $\hat k\in\N$ and $k\geq \hat k$, we assume \eqref{eq:xkykUpperBound} holds and 
	     $
	         \mathcal{S}_{k-1}\leq \left(\frac{\kappa}{1-\theta}\right)^{\tfrac{2}{1-2\theta}}.
	   $
	     Two cases are recognized: (a) $\theta\in(0,\nicefrac{1}{2}]$; (b) $\theta\in(\nicefrac{1}{2},1)$. In Case (a), if $\theta\in(0,\nicefrac{1}{2})$, then $\psi(\mathcal{S}_{k-1})\leq\sqrt{\mathcal{S}_{k-1}}$. For $\theta=\nicefrac{1}{2}$, we get $\psi(\mathcal{S}_{k-1})=\tfrac{\kappa}{1-\theta}\sqrt{\mathcal{S}_{k-1}}$, which implies $\max\set{\sqrt{\mathcal{S}_{k-1}},\psi(\mathcal{S}_{k-1})}=\max\set{1,\tfrac{\kappa}{1-\theta}} \sqrt{\mathcal{S}_{k-1}}$. Then, $\max\set{\sqrt{\mathcal{S}_{k-1}},\psi(\mathcal{S}_{k-1})}\leq\max\set{1,\tfrac{\kappa}{1-\theta}} \sqrt{\mathcal{S}_{k-1}}$. In Case (b), it holds that 
	     $\psi(\mathcal{S}_{k-1})\geq\sqrt{\mathcal{S}_{k-1}}$,
	     i.e., $\max\set{\sqrt{\mathcal{S}_{k-1}},\psi(\mathcal{S}_{k-1})}= \tfrac{\kappa}{1-\theta} \mathcal{S}_{k-1}^{1-\theta}$. Combining both cases, for all $k\geq\overline k:=\max \set{\tilde k, \hat k}$, we end up with
	     \begin{equation}\label{eq:xkUpperBound}
	         \|x_i^k-x_i^\star\|\leq \left\{
	         \begin{array}{ll}
	             \nu \max\set{1,\tfrac{\kappa}{1-\theta}} \sqrt{\mathcal{S}_{k-1}} &~~~ \mathrm{if}\  \theta\in(0,\nicefrac{1}{2}], \vspace{1mm}\\
	             \nu \tfrac{\kappa}{1-\theta} \mathcal{S}_{k-1}^{1-\theta} &~~~ \mathrm{if}\ \theta\in(\nicefrac{1}{2},1).
	         \end{array}
	         \right. 
	     \end{equation}
	    
	On the other hand,     
	it follows from \Cref{cor:descentProp} that
	     \begin{align*}
	     \begin{array}{ll}
	        & \mathcal{S}_{k-1}-\mathcal{S}_{k}\\
	        &=\mathcal{L}(\bm x^{k-1},\bm x^{k-2})-\mathcal{L}(\bm x^{k},\bm x^{k-1}) \geq \varrho \sum_{i=1}^N \left(\D_{h}(\bm x^{k-1,i},\bm x^{k-1,i-1}+\D_{h}(\bm x^{k-2,i},\bm x^{k-2,i-1})\right)\\
	         &\geq \tfrac{\varrho}{2N} \left(\sqrt{\D(\bm x^{k-1,i},\bm x^{k-1,i-1})}+\sqrt{\D(\bm x^{k-2,i},\bm x^{k-2,i-1})}\right)^2\\
	         &\geq \tfrac{\varrho}{2N\widetilde{c}^2}\|(\mathcal{G}_1^{k},\ldots,\mathcal{G}_N^{k})\|^2
	         \geq \tfrac{\varrho}{2N\widetilde{c}^2} \dist(0,\partial \mathcal{L}(\bm x^k,\bm x^{k-1}))^2\geq \tfrac{\varrho}{2N\widetilde{c}^2\kappa^2} \mathcal{S}_{k-1}^{\theta}=c_2 \mathcal{S}_{k-1}^{\theta},
	         \end{array}
	     \end{align*}
	    where $c_2:= \tfrac{\varrho}{2N\widetilde{c}^2\kappa^2}$. The results then follow from $\mathcal{S}_k\to 0$,  \eqref{eq:xkUpperBound} and \Cref{fac:convRate1}. 
	 \end{proof}
		\vspace{-4mm}	
	\section{Application to symmetric nonnegative matrix tri-factorization}\label{sec:SNMF}
	
	A natural way of analyzing large data sets is finding an effective way to represent them using dimensionality reduction methodologies.
	\DEF{Nonnegative matrix factorization} (NMF) is one such technique that has received much attention
	in the last few years; see, e.g., \cite{cichocki2009nonnegative,gillis2014nonnegative} and the references therein. In order to extract hidden and important features from data, NMF decomposes the data matrix into two factor matrices (usually much smaller than the original data matrix) by imposing componentwise nonnegativity and (possibly) other  constraints such as sparsity to take prior information into account. 
	More precisely, let the data matrix be $X=[x_1,x_2,\ldots,x_n]\in\R_+^{m\times n}$ where each $x_i$ represents some data point. 
	NMF seeks a decomposition of $X$ into a nonnegative $n\times r$ basis matrix $U=[u_1,u_2,\ldots,u_r]\in\R_+^{m\times r}$ and a  nonnegative $r\times n$ coefficient matrix $V=[v_1,v_2,\ldots,v_r]^T\in\R_+^{r\times n}$ such that 
	\begin{equation}\label{eq:nmfEq}
	    X\approx UV,    
	\end{equation}
	where $\R_+^{m\times n}$ is the set of $m\times n$ nonnegative matrices. Extensive research has been carried out on variants of NMF, and most studies have focused on algorithmic developments, but with very limited convergence theory. This motivates us to study the application of \refBIBPA[] to a variant of NMF, namely SymTriNMF; 
	see~\eqref{eq:usnmtf} for the formulation of SymTriNMF as an optimization problem.  
	
	One popular application of SymTriNMF is community detection. 
	Let $X$ be the adjacency matrix of graph so that $X_{ij} = 1$ if item $i$ is connected to item $j$, and $X_{ij}= 0$ otherwise. 
	Let also $X \approx UVU^T$ be a SymTriNMF decomposition of $X$. Each column of $U$ corresponds to a community, that is, to a subset of items highly connected. In other words, the entry $U_{jk}$ of $U$ indicates the membership of item $j$ within community $k$, 
	and $U_{jk} > 0$ if $j$ belongs to community $k$.  
	The $r$-by-$r$ matrix $V$ indicates the relationship between communities, that is, whether the items within two communities are likely to interact: 
	 $V_{kp}$ is the "strength" of the interaction between the $k$th and $p$th communities. We have 
	$ X \approx \sum_{k=1}^r \sum_{p=1}^r U_{:k} V_{k,p} U_{:p}^T, $
	so that $X$ is decomposed via the sum of $r^2$ rank-one factors corresponding to the $r$ communities and their interactions; see~\cite{wang2011simultaneous,zhang2012overlapping} for more details. Note that SymTriNMF is closely related to the mixed membership stochastic blockmodel~\cite{airoldi2008mixed}.  
	
	Given $U^k$ and $V^k$, we next derive the closed-form solutions for $U^{k+1}$ and $V^{k+1}$.
	
	\begin{thm}[closed-form solutions of the subproblem \eqref{eq:xik1} for SymTriNMF]
	\label{thm:iterBIBPA}
	    Let $h_1$ and $h_2$ be the kernel functions given in \eqref{eq:kernelh4h21} and \eqref{eq:kernelh4h22} and $U^k$ and $V^k$ are given. Then,
	    \begin{enumerate}
	        \item \label{thm:iterBIBPA1}
	        the iteration $U^{k+1}$ of the subproblem \eqref{eq:xik1} is given by
	            \begin{equation}\label{eq:uk01}
	                U^{k+1}= \tfrac{1}{t_k} \max\set{\frac{1}{\gamma_1^k }\big(\nabla_U h_1(U^k,V^k)-\gamma^k_1\nabla_U f(U^k,V^k)+\alpha_1^k(U^k-U^{k-1})\big),0}
	            \end{equation}
	            with 
	            \begin{align*}
	                      		\nabla_U f(U^k,V^k)&=-XU^k(V^k)^T-X^TU^kV^k+U^kV^k(U^k)^TU^k(V^k)^T+U^k(V^k)^T(U^k)^TU^kV^k,\\
	            		 \nabla_U h_1(U^k,V^k)&=\left(a_1 \|U^k\|_F^2 \|V^k\|_F^2+b_1 (\|X\|_F~\|V^k\|_F + \varepsilon_1) \right)U^k,
	            \end{align*}
			 and
			\begin{align}\label{eq:ckSol}
			t_k=	 \frac{\tau_1}{3}+\sqrt[3]{\frac{\tau_2+\sqrt{\Delta_1}}{2}+\frac{\tau_1^{3}}{27}}+\sqrt[3]{\frac{\tau_2-\sqrt{\Delta_1}}{2}+\frac{\tau_1^{3}}{27},}
			\end{align}
		where $	\tau_1=b_1(\|X\|_F \|V^k\|_F+\varepsilon_1), \quad
		\tau_2=a_1\|V^k\|_F^2\|\max\set{G^k,0}\|_F^2, \quad
		\Delta_1=\tau_2^2 +\frac{4}{27}\tau_2^2\tau_1^3  
		$ with
		$
		G^k:=\tfrac{1}{\gamma^k_1}\left(\nabla_U h_1(U^k,V^k)-\gamma^k_1\nabla_U f(U^k,V^k)+\alpha_1^k(U^k-U^{k-1})\right). 
		$

	        \item \label{thm:iterBIBPA2}
	        for  $\eta_k:=a_2\|U^{k+1}\|^4+\varepsilon_2$, the iteration $V^{k+1}$ of the subproblem \eqref{eq:xik1} is given by
	            \begin{equation}\label{eq:vk01}
	                 V^{k+1}= \max\set{V^k-\tfrac{1}{\eta_k}\Big(\alpha_2^k(V^k-V^{k-1})-\gamma_2^k\nabla_V f(U^{k+1},V^k)\Big),0},
	            \end{equation}
	            with $\nabla_V f(U^{k+1},V^k)=(U^{k+1})^T X U^{k+1} + (U^{k+1})^T U^{k+1} V^k (U^{k+1})^T U^{k+1}.$
	        
	    \end{enumerate}
	    \end{thm}
	
	\begin{proof}
		Setting $g_1:=\delta_{U\geq 0}$ and $f(U,V)=\tfrac{1}{2}\|X-UVU^T\|_F^2$, it follows from \eqref{eq:xik1} that
		\begin{equation}\label{eq:uk1proof}
			\begin{array}{ll}
			&U^{k+1}=  \argmin_{U\in\R^{m\times r}} \Big\{\innprod{\nabla_U f(U^k,V^k)-\tfrac{\alpha_1^k}{\gamma^k_1}(U^k-U^{k-1})}{U-U^k} \\
			&~~~~~~~~~~~~~~~~~~+\tfrac{1}{\gamma_1^k}\mathbf D_{h_1}((U,V^k),(U^k,V^k))+ g_1(U)\Big\}\\
			&=\argmin_{U\geq 0} \Big\{\tfrac{1}{\gamma^k_1}\innprod{\gamma^k_1\nabla_U f(U^k,V^k)-\nabla_U h_1(U^k,V^k)-\alpha_1^k(U^k-U^{k-1})}{U} +\tfrac{1}{\gamma_1^k}h_1(U,V^k)\Big\}.
			\end{array}
		\end{equation}
		By \cite[Corollary 3.5]{tam2017regularity}, the normal cone of the nonnegativity constraint $U\geq 0$ is 
		$
		\mathcal{N}_{U\geq 0}(U^k)=\set{P\in\R^{m\times r}\mid U^k\odot P=0,\ P\leq 0}
		$
		where $U^k\odot P$ denotes the \DEF{Hadamard products} given pointwise by $(U^k\odot P)_{ij}:=U_{ij}^kP_{ij}$ for $i\in{1,\ldots,m}$ and $j\in{1,\ldots,r}$. The first-order optimality conditions for the subproblem \eqref{eq:uk1proof} yields that 
		$G^k-(a_1\|U^{k+1}\|_F^2 \|V^k\|_F^2+b_1 (\|X\|_F \|V^k\|_F+\varepsilon_1))U^{k+1}\in\mathcal{N}_{U\geq 0}(U^{k+1})$. 
		
		We now consider two cases: (i) $G_{ij}\leq 0$; (ii) $G_{ij}> 0$. In Case (i), we have 
		$$P_{ij}=G_{ij}^k-(a_1\|U^{k+1}\|_F^2\|V^k\|_F^2+b_1 (\|X\|_F \|V^k\|_F+\varepsilon_1))U_{ij}^{k+1}\leq 0,$$ 
		hence $U_{ij}^{k+1}=0$.
	In Case (ii), if $U_{ij}^{k+1}=0$, then $P_{ij}=G_{ij}^k>0$, which contradicts $P\leq 0$; hence  $G_{ij}^k-(a_1\|U^{k+1}\|_F^2\|V^k\|_F^2+b_1 (\|X\|_F \|V^k\|_F+\varepsilon_1))U_{ij}^{k+1}=0$. 
	Combining both cases, we get
		$
			(a_1\|U^{k+1}\|_F^2\|V^k\|_F^2+b_1 (\|X\|_F \|V^k\|_F+\varepsilon_1))U^{k+1}=\proj_{G\geq 0}(G^k).
		$
	Denote $t_k=	a_1\|U^{k+1}\|_F^2\|V^k\|_F^2+b_1 \|X\|_F \|V^k\|_F$, then $ \|U^{k+1}\|_F^2 = (t_k- b_1 \|X\|_F \|V^k\|_F)/(a_1 \|V^k\|_F^2)$.  We have $ t_k^3-b_1 \|X\|_F \|V^k\|_F t_k^2-a_1  \|V^k\|_F^2 \|\proj_{G\geq 0}(G^k)\|_F^2=0.  $
	Note that the third order polynomial equation $y^{2}(y-a)=c$ has the unique real solution $y=\frac{a}{3}+\sqrt[3]{\frac{c+\sqrt{\Delta}}{2}+\frac{a^{3}}{27}}+\sqrt[3]{\frac{c-\sqrt{\Delta}}{2}+\frac{a^{3}}{27},}$  where $\Delta=c^{2}+\frac{4}{27}ca^{3}$. Then we get \eqref{eq:ckSol}. Finally, the result follows from $U^{k+1}=\frac{\proj_{G\geq 0}(G^k)}{t_k} $.
	
	By setting $g_2:=\delta_{V\geq 0}$ and invoking \eqref{eq:xik1}, we get
		\begin{equation*}
		\label{eq:vk1proof}
			\begin{array}{ll}
			& V^{k+1}=  \argmin_{V\in\R^{r\times r}} \Big\{\innprod{\nabla_V f(U^{k+1},V^k)-\tfrac{\alpha_2^k}{\gamma_2^k}(V^k-V^{k-1})}{V-V^k} \\
			&~~~~~~~~~~~~~~~~~~+\tfrac{1}{\gamma^k_2}\mathbf D_{h_2}((U^{k+1},V),(U^{k+1},V^k))+ g_2(V)\Big\}\\
			&=\argmin_{V\geq 0}\frac{1}{\gamma_2^k} \innprod{\gamma_2^k\nabla_V f(U^{k+1},V^k)-\alpha_2^k(V^k-V^{k-1})-\nabla h_2(U^{k+1},V^k)}{V}+ \frac{1}{\gamma_2^k}h_2 (U^{k+1},V) \\
			&=\argmin_{V\geq 0} 
			\Big\{
			\big\|V-\frac{1}{a_2\|U^{k+1}\|^4+\varepsilon_2}\big( \alpha_2^k(V^k-V^{k-1})+ \nabla h_2(U^{k+1},V^k) -\gamma_2^k\nabla_V f(U^{k+1},V^k) \big)\big\|_F^2 		\Big\}\\
			&= \proj_{V\geq 0}\left(V^k-\tfrac{1}{\eta_k}\big(\alpha_2^k(V^k-V^{k-1})-\gamma_2^k\nabla_V f(U^{k+1},V^k)\big)\right),
			\end{array}
		\end{equation*}
	which proves \eqref{eq:vk01}.
	\end{proof}

		\vspace{-4mm}	
	\section{Final remarks}\label{sec:conclusion} 
	The descent lemma is a key factor for analyzing the first-order methods in both Euclidean and non-Euclidean settings. Owing to the notion of block relative smoothness, it was shown that the descent lemma is still valid for each block of variables for structured nonsmooth nonconvex problems with non-Lipschitz gradients. Based on this development, \refBIBPA[] was introduced to deal with such problems, and it was shown to be globally convergent for K{{\L}} functions and its convergence rate was also studied. Besides, it was shown that the objective of the symmetric nonnegative matrix tri-factorization (SymTriNMF) problem is block relatively smooth, and the corresponding subproblems can be solved in closed forms. To our knowledge, \refBIBPA[] is the first algorithm with rigorous theoretical guarantee of convergence for this problem. We emphasize that the main objective of this paper is to provide a theoretical and algorithmic framework that can handle block structured nonsmooth nonconvex problems under the block relative smoothness assumption. Hence, a comprehensive numerical experiments  for such structured problems are postponed to a future work. 
    
	\ifarxiv
		\bibliographystyle{plain}
	\else
		\phantomsection
		\addcontentsline{toc}{section}{References}
		\bibliographystyle{spmpsci}
	\fi
	\bibliography{Bibliography}

\end{document}